\documentclass[a4paper, 11pt]{article}
\usepackage[headings]{fullpage}
\usepackage[utf8]{inputenc}
\usepackage[T1]{fontenc}
\usepackage{amsmath,amsfonts,tikz,wrapfig}
\usepackage{amsthm,latexsym,amscd,amssymb}
\usepackage[all]{xy}
\usepackage{verbatim}
\usepackage{graphics,epsfig,color}
\usepackage{mathrsfs} 
\usepackage{rotating}
\usepackage{hyperref}
\usepackage{tikz,tikz-3dplot}

\newcommand*\circled[1]{\tikz[baseline=(char.base)]{
            \node[shape=circle,draw,inner sep=2pt] (char) {#1};}}

\newtheoremstyle{mytheoremstyle} 
    {0.6cm}                    
    {0.6cm}                    
    {\itshape}                   
    {}                           
    {\bfseries}                   
    {.}                          
    {.5em}                       
    {}  

\newtheoremstyle{mytheoremstarstyle} 
    {0.6cm}                    
    {0.6cm}                    
    {\itshape}                   
    {}                           
    {\bfseries}                   
    {.}                          
    {.5em}                       
    {\textbf{\thmname{#1}\thmnote{ (#3)}}}  
  
\newtheoremstyle{mydefinitionstyle} 
    {0.6cm}                    
    {0.6cm}                    
    {\normalfont}                   
    {}                           
    {\bfseries}                   
    {.}                          
    {.5em}                       
    {}  
    
\newtheoremstyle{myremarkstyle} 
    {0.6cm}                    
    {0.6cm}                    
    {\normalfont}                   
    {}                           
    {\itshape}                   
    {.}                          
    {.5em}                       
    {}  

\theoremstyle{mytheoremstarstyle}
  
\newtheorem{theorem*}{Theorem}[section]

\theoremstyle{mytheoremstyle}

\newtheorem{theorem}{Theorem}[section]

\newtheorem{lemma}[theorem]{Lemma}
\newtheorem{proposition}[theorem]{Proposition}
\newtheorem{assumption}[theorem]{Assumption}
\theoremstyle{mydefinitionstyle}\newtheorem{definition}[theorem]{Definition}
\theoremstyle{myremarkstyle}\newtheorem{remark}[theorem]{Remark}

\def\Hom{\mathrm{Hom}}

\def\bbZ{\mathbb{Z}}\def\bbR{\mathbb{R}}\def\bbQ{\mathbb{Q}}

\def\colim{\mathop{\rm colim}}

\begin{document}

\title{Products in Generalized Differential Cohomology}
\author{Markus Upmeier}
\date{\today}
\maketitle

\begin{abstract}
It is shown in this paper that multiplicative cohomology theories $E$ that are rationally even -- a technical condition that is often satisfied --
the Hopkins-Singer construction of generalized differential cohomology $\hat{E}$ has a 
unital, graded commutative multiplicative structure. To this end, a more explicit integration and a differential cohomology theory for pairs are also developed.
\end{abstract}


\section{Introduction}
The idea of differential cohomology theory is to combine geometric and cohomological information of manifolds. The many historical examples are embraced by the following description by \cite{bunke_schick_uniqueness}:
Let $E$ be a generalized cohomology theory and let $V^*=E^*(pt)\otimes_{\bbZ} \bbR$ denote
the graded coefficient vector space. A \emph{differential extension} of $E$ is a contravariant functor $\hat{E}^*$ from smooth manifolds to graded Abelian groups together
with natural transformations
$\Omega^{*-1}(-;V)/im(d)	\overset{a}{\Longrightarrow}	\hat{E}^*$,
$\hat{E}^*		\overset{R}{\Longrightarrow}	\Omega^*_{cl}(-;V)$, and
$\hat{E}^*		\overset{I}{\Longrightarrow} E^*$
such that for any smooth manifold $M$ the following diagram commutes and has an
exact upper horizontal line:
$$\xymatrix{
E^{*-1}(M)\ar[r]^-{ch}	&	\Omega^{*-1}(M;V)/im(d)\ar[r]^-a\ar[rd]_d	&	\hat{E}^*(M)\ar[r]^-I\ar[d]_R	&	E^*(M)\ar[r]\ar[d]^{ch}	&	0\\
&&\Omega_{cl}^*(M;V)\ar[r]	&	H^*(M;V)
}$$
Such a differential extension is called \emph{multiplicative} if $\hat{E}$
is a functor into graded commutative rings, if $I$ and $R$ are unital ring
homomorphisms, and if
$$ a(\Theta)\cup \hat{x}=a(\Theta\wedge R(\hat{x}))\hspace{1cm} \forall\Theta\in \Omega^{n-1}(M;V),\, \hat{x}\in \hat{E}^m(M). $$

Historically, the first examples were for $E=H\mathbb{Z}$ and appeared as the sheaf-theoretic Deligne cohomology (see for instance \cite{gajer}) and as Cheeger-Simons differential characters.  These were defined in \cite{cheeger_simons} to provide a natural home for secondary invariants that can take more geometry of say a Riemannian manifold into account. Much later, another model for ordinary differential cohomology was introduced in \cite{bunke_kreck_schick} using stratifolds. 
The question whether the various constructions of ordinary differential cohomology
yield isomorphic theories was answered by \cite{simons_sullivan} and subsequently the case of generalized cohomology was dealt with in \cite{bunke_schick_uniqueness}, establishing criteria when this is true.

Differential refinements of $K$-theory were studied by \cite{lottRZ} and \cite{bunke_schick_k} with which most notably a refinement of the families index theorem may be proven \cite{lott}. On the other hand, \cite{freed_charge_quant} and \cite{freed_RRfields} exhibit the importance of differential $K$-theory in mathematical physics, where it may be used to encode charge quantization phenomena.

Another broad class of examples was defined in \cite{bunke_schick_schroeder_wiethaup} were it is shown that every Landweber exact cohomology theory may be refined to a \emph{multiplicative} differential cohomology theory.
Finally, \cite{hopkins_singer} use abstract homotopy theory to construct differential refinements of \emph{arbitrary} generalized cohomology theories. Unfortunately, starting with a multiplicative cohomology theory, this general construction does not have obvious multiplicative properties, which is what we will take up in the present paper.

Roughly speaking, the construction in \cite{hopkins_singer} starts by representing $E$
by an $\Omega$-spectrum $E_n$ and choosing a refinement of the generalized Chern
character
\[
	E^n(X)=[X,E_n] \xrightarrow{ch} H^n(X;V),\; [f]\mapsto f^*[\iota_n],\hspace{5ex} \text{where } [\iota_n]=ch(\mathrm{id}_{E_n}),
\]
to maps and cochains $E_n^X \rightarrow C^n(X;V)$. This is not canonical, and the different possibilities are the choices of fundamental cocycles $\iota_n\in C^n(E_n;V)$ representing the fundamental cohomology classes $[\iota_n]$.
At this point many difficulties arise since different choices of fundamental cocycles lead to \emph{non-canonically} isomorphic differential extensions. To carry over a construction, like a product, to differential cohomology all the structure needs to be refined in a manner
compatible with the fundamental cocycles. We will deal with this issue by cutting down the non-uniqueness in even dimensions and demanding a compatibility with the so-called integration map. That is, we will work with cohomology theories that are \emph{rationally even}, meaning that $E^*(pt)$ is torsion in all odd degrees. This is a class of cohomology theories that contains
most of the cohomology theories of interest, e.g.
oriented cobordism $MSO$, unoriented $MO$ cobordism, and complex cobordism $MU$, stable cohomotopy theory $\pi_s^*$, ordinary cohomology $HA$ with coefficients in an Abelian group $A$, real, complex, and quaternionic $K$-theory, Brown-Peterson cohomology $BP$,
Morava $K$-theory $K(n)$, John-Wilson theory, and elliptic cohomology theories (but not algebraic $K$-theory -- thanks to U. Bunke for pointing this out to me).
\\
%


The main results of this paper are the following two theorems:

\begin{theorem*}
Let $E$ be a rationally even, multiplicative cohomology theory. Then, there exists a differential extension $\hat{E}^*$ to a multiplicative differential cohomology theory. Moreover, the product structure is compatible with the integration map.
\end{theorem*}

Here, by an \emph{integration map} on $\hat{E}$ we mean a natural transformation $\int_{S^1}$
from $\hat{E}^{*+1}(S^1\times-)$ to $\hat{E}^*$ which lifts the following map on generalized cohomology
$$ \int_{S^1}: E^{n+1}(S^1\times M, 1\times M)\cong \tilde{E}^{n+1}(\Sigma M^+)\cong \tilde{E}^n(M^+)=E^n(M) $$
and is compatible with $I$, $R$, and $a$ and satisfies, in addition, certain reasonable assumptions, compare with section 3.

\begin{theorem*}
For closed submanifolds $N\subset M$ there is a differential cohomology of pairs $\hat{E}^*(M,N)$ with the expected properties;
there is a natural long exact sequence
$$
\xymatrix@R=0.5cm{
\cdots\ar@{-}[r] &\hat{E}^{n-1}_{flat}(N)\ar[r]\ar[r]^-{\delta_1}  & \hat{E}^n(M,N)\ar[r]^{i^*}	&\hat{E}^n(M)\ar[r]^{j^*}	&  \hat{E}^n(N)\ar[r]^-{\delta_2}
& E^{n+1}(M,N)\ar@{-}[r]&\cdots
 }
 $$
\end{theorem*}


\subsubsection*{Notation}

Let $A\subset X$.
Singular chains $C_n(X,A)$, cochains $C^n(X,A)$, and cohomology are understood to have real coefficients. Cocycles and coboundaries will be denoted by $Z^n(X,A)$ and $B^n(X,A)$ respectively. We will make deliberate use of the Acyclic Models Theorem. Reduced cohomology and cochains will be indicated by an upper tilde. The space of closed differential forms will by denoted by $\Omega^n_{cl}(M)$.
For a graded vector space $V$ set
\begin{align*}
C_n(X,A; V)&=\bigoplus\limits_{i+j=n} C_i(X,A;V^j),\\
C^n(X,A; V)&=\prod\limits_{i+j=n} C^i(X,A;V^j)=\Hom{}_\bbR(C_n(X,A;V), \bbR),\\
\Omega^n(M;V)&=\prod\limits_{i+j=n} \Omega^i(M;V^j)=\bigoplus\limits_{i+j=n} \Omega^i(M;V^j).
\end{align*}
If $M$ is a manifold and $i: N\subset M$ a closed submanifold we will write $C^n_{s}(M,N)$ for the complex of smooth cochains, i.e. homomorphisms on smooth chains $C_n^{s}(M)=\bigoplus\limits_{\sigma: \Delta^n\rightarrow M \text{ smooth}} \bbR$ which are zero on $C_n^{s}(N)$.
The \emph{relative deRham complex} is defined by $\Omega^*(M,N)=\left\{ \omega\in\Omega^*(M)\, \middle|\; i^*\omega=0 \right\}$. We then have a short exact sequence
\begin{equation}\label{eqn:KES-forms}
0\rightarrow \Omega^*(M,N)\rightarrow \Omega^*(M)\rightarrow \Omega^*(N)\rightarrow 0
\end{equation}

Fix the standard \emph{Eilenberg-Zilber chain equivalence} $B: C_*(I)\otimes C_*(X)\rightarrow C_*(I\times X)$ coming from the usual subdivision of the prism. The \emph{fundamental $1$-chain} is $[I]: \Delta^1\rightarrow I\in C_1(I),\;\; (t_0,t_1)\mapsto t_1$. \emph{Integration of cochains along the interval} is the map
$\int_I$ from $C^n(I\times X)$ to $C^{n-1}(X)$ given by
$$C^n(I\times X) \ni u\longmapsto \left( C_{n-1}(X)\overset{[I]\otimes id}{\longrightarrow} C_1(I)\otimes C_{n-1}(X)\overset{B}{\rightarrow}C_n(I\times X)\overset{u}{\longrightarrow}\bbR  \right)$$
This may be extended componentwise to cochains with coefficients in a graded vector space and extends the corresponding map for differential forms (cf. \cite{hopkins_singer}, p. 32). By pulling back along the canonical maps we obtain also integration maps for cochains on $S^1\times X$ and on the suspension $\Sigma X=S^1\wedge X$. The following two formulas will be imporant in the sequel:
\begin{equation}\label{eqn:coboundary-integral-formula}
\int_I\delta u+\delta \int_I u=i_1^*u-i_0^*u,\hspace{1cm} u\in C^n(I\times (X,A);V),
\end{equation}
and the \emph{pullback formula} for continuous maps $c:(X,A)\rightarrow (Y,B)$
\begin{equation}\label{eqn:pullback-formula}
\int_I (id_I\times c)^*u=c^*\int_I u,\hspace{1cm} u\in C^n(I\times (Y,B); V).
\end{equation}

\section{Generalized Differential Cohomology for Pairs}
In this section we will generalize the construction of \cite{hopkins_singer} to pairs, exhibit the
mentioned long exact sequence, and review the functorial properties of differential cohomology. Next, we will establish certain ``canonical maps'' which are needed for the construction of products, as well as some technical results for manipulating differential cocycles.

\subsection{The Fundamental Cocycle \& Construction of $\hat{E}$ for Pairs}

Let $E$ be a cohomology theory represented by an $\Omega$-spectrum $(E_n,\varepsilon_n)$, i.e.
a sequence of pointed topological spaces $pt\in E_n$ together with pointed homeomorphisms
$$ \varepsilon_{n-1}^{adj}: E_{n-1}\overset{\approx}{\longrightarrow} \Omega E_n. $$
We have the \emph{generalized Chern character}
$ch: \tilde{E}^*\Longrightarrow \tilde{E}^*\otimes_\bbZ \bbR \overset{\cong}{\Longrightarrow} \tilde{H}^*(-;V)$
which by Yoneda's Lemma is implemented by \emph{fundamental cohomology classes} $[\iota_n]$.
The $[\iota_n]$ are related by the suspension and the structure maps and may therefore be viewed as an element of
$ \lim\limits_{\overset{\longleftarrow}{n}} \tilde{H}^n(E_n;V)$,
the limit being taken over
$ \tilde{H}^{*+n}(E_n;V)\overset{\varepsilon_{n-1}^*}{\longrightarrow} \tilde{H}^{*+n}(\Sigma E_{n-1};V)\overset{susp}{\cong} \tilde{H}^{*+n-1}(E_{n-1};V)$.
We claim that this vector space may be identified with the $0$-th
cohomology of the cochain complex
$ \lim\limits_{\overset{\longleftarrow}{n}} C^{*+n}(E_n,pt;V)$,
where the limit is now taken over
\[
C^{*+n}(E_n,pt;V)\underset{\varepsilon_{n-1}^*}{\longrightarrow} C^{*+n}(\Sigma E_{n-1},pt;V) \underset{\int_{S^1}}{\longrightarrow} C^{*+n-1}(E_{n-1},pt;V).
\]

\begin{proof}[Proof of claim]
The suspension isomorphism can be written as multiplication with $[S^1]$ from the left.
Using universal coefficients, that $\Hom_\bbR$ takes colimits in the first variable to limits, the fact that two limits commute, and that directed colimits and homology commute we deduce:
\begin{align*}
	H^n\Big( \lim\limits_{\overset{\longleftarrow}{k}} &C^{*+k}(E_k,pt;V)  \Big)=H^n\Big(\lim\limits_{\overset{\longleftarrow}{k}} \prod_{j\in\bbZ} C^{*-j+k}(E_k,pt;V^{j})  \Big)
	=H^n\Big( \prod_{j\in\bbZ } \lim\limits_{\overset{\longleftarrow}{k}} C^{*-j+k}(E_k,pt;V^{j})  \Big)\\
	&=\prod_{j\in\bbZ} H^n\Big(\lim\limits_{\overset{\longleftarrow}{k}} C^{*-j+k}(E_k,pt;V^{j})  \Big)
	=\prod_{j\in\bbZ} \Hom{}_\bbR \Big( H_n(\colim C_{*-j+k}(E_k,pt)), V^{j}\Big)\\
	&=\prod_{j\in\bbZ} \Hom{}_\bbR \Big( \colim H_n(C_{*-j+k}(E_k,pt)), V^{j}\Big)
	=\prod_{j\in\bbZ} \lim\limits_{\overset{\longleftarrow}{k}}\Hom{}_\bbR \Big(H_n(C_{*-j+k}(E_k,pt)), V^{j}\Big)\\
	&=\lim\limits_{\overset{\longleftarrow}{k}} \prod_{j\in\bbZ} H^n(C^{*-j+k}(E_k,pt; V^j))=\lim\limits_{\overset{\longleftarrow}{k}} H^n\Big(\prod_{j\in\bbZ} C^{*-j+k}(E_k,pt; V^j)\Big)\\
	&=\lim\limits_{\overset{\longleftarrow}{k}} H^n( C^{*+k}(E_k, pt ;V) )=\lim\limits_{\overset{\longleftarrow}{k}} \tilde{H}^{n+k}(E_k;V)
\end{align*}

\end{proof}

\noindent
If follows that we may choose \emph{fundamental cocycles} $\iota_n \in Z^n(E_n,pt;V)$ with the property that
\begin{equation}\label{eqn:choice-fund-cocycle}
\iota_{n-1}=\int_{S^1} \varepsilon_{n-1}^*\iota_n.
\end{equation}

\begin{definition}
The \emph{$n$-th differential $E$-cohomology} $\hat{E}^n(M,N)$ or, when emphasizing the dependence on the fundamental cocycle, $\hat{H}^n\left((M,N); (E_n,\iota_n,V)  \right)$ is the set of equivalence classes of triples
\begin{align*}
(M,N)\overset{c}{\longrightarrow} (E_n,pt),\;\;\;\;
&\omega \in \Omega_{cl}^n(M,N; V),\;\;\;\;
h\in C^{n-1}_{s}(M,N;V)\hspace{0.5cm}\text{such that}\\
\delta h &= \omega - c^*\iota_n \in C^n_{s}(M,N;V)
\end{align*}
modulo the equivalence relation that $(c_0,\omega_0,h_0)\sim (c_1,\omega_1,h_1)$ iff $\omega_0=\omega_1$ and
\begin{align*}
\exists &I\times 
(M,N) \overset{C}{\longrightarrow} (E_n,pt):\; c_0\simeq c_1\\
	   \exists &H\in C^{n-1}_{s}( I\times (M,N) ;V )\text{ with } i_0^*H=h_0,\, i_1^*H=h_1\\
	   \text{such that }&\delta H=pr^*\omega - C^*\iota_n\in C^n_{s}(I\times (M,N); V).
\end{align*}
\end{definition}
Note that this definition makes sense more generally for
any graded real vector space $V$, pointed topological space $E$, and
reduced singular cocycle $\iota\in Z^n(E,pt;V)$.
The above relation is obviously reflexive and symmetric. Transitivity requires more work:
Suppose $(c_0,\omega,h_0)\sim (c_1,\omega,h_1)\sim (c_2,\omega,h_2)$ via
$(H_0,C_0)$ and $(H_1,C_0)$. By pulling back along a smooth map $I\rightarrow I$
that increases from $0$ to $1$ and is constant on $[0,1/4]$ and $[3/4,1]$ we may
assume that $H_0|_{[0,1/4]}=pr^*h_0, C|_{[0,1/4]}=c_0\circ pr$ and similar on $[3/4,1]$
and for $(H_1,C_1)$. Then, glue $H_0, H_1$ together:

\begin{lemma}\label{lem:gluing}
Let $A \cup B =X$ be an open covering and $\tilde{A}\subset A, \tilde{B}\subset B$ be subsets with
$\tilde{A}\cap B=\tilde{B}\cap A$. Then, restriction
$
	C^n(A\cup B, \tilde{A}\cup\tilde{B})\longrightarrow C^n(A,\tilde{A})\times C^n(B,\tilde{B})
$
is surjective.
\end{lemma}
\begin{proof}
Let $u\in C^n(A,\tilde{A}), v\in C^n(B,\tilde{B})$ and
consider the projection
$\pi: \Delta^{n+1}\rightarrow \Delta^n$, $\pi(t_0,\ldots,t_{n+1})= \left(t_0+\frac{t_{n+1}}{n+1},\ldots, t_n+\frac{t_{n+1}}{n+1}\right)$.
Define a subdivision operator by
$$ S: C_n(X)\rightarrow C_n(X),\hspace{2ex} (\Delta^n \overset{\sigma}{\longrightarrow} X) \mapsto (-1)^n \partial (\sigma\circ \pi)+\sigma$$
$\pi\circ d^i: \Delta^n\rightarrow \Delta^n$ parameterizes a subset of $\Delta^n$ of area $\frac{area(\Delta^n)}{n+1}$.
Therefore, after a finite \emph{minimal} number $m=m(\sigma)$ of applications of $S$ any simplex $\sigma \in C_n(X)$ will be a chain $\sum\limits_k n_k\cdot \tau_k$ consisting only of simplices $\tau_k$ whose image lies entirely in $A$ or entirely in $B$. We define
$w(\sigma)$ as $\sum_k n_k \begin{cases}u(\tau_k)\text{ if }\tau_k(\Delta^n)\subset A,\\v(\tau_k)\text{ if }\tau_k(\Delta^n)\subset B.\end{cases}$ By minimality, $w$ restricts to $u$ and $v$. For the last statement we remark that our subdivision operator takes the subcomplex of smooth chains to itself.
\end{proof}

Now, apply the above lemma to $[0,2]\times (M,N)$ with open cover
$[0,1[\times(M,N)$, $]\frac 34, \frac54[\times(M,N)$, $]1,2]\times(M,N)$
and the cochains $H_0, pr^*h_1, H_1$ translated properly to obtain $H$. The cochain $H$ clearly restricts to $h_0$ and $h_2$ at the endpoints. Let $C=C_0* C_1$ be the composition of homotopies. To show $(c_0,\omega,h_0)\sim (c_1,\omega,h_1)$ it remains to verify
$ \delta H = pr^*\omega-C^*\iota_n$
for which it suffices to consider a smooth simplex $\sigma$ with image entirely contained in one of the three open subsets -- since $S$ leaves cocycles invariant. For example if $\sigma(\Delta^n)\subset [0,1[\times M$ then $\sigma=i_*\sigma$ for the inclusion $i:[0,1[\times M \subset [0,2]\times M$. We have
\begin{align*}
(\delta H)(\sigma)&=(\delta i^*H)(\sigma)=(\delta H_0)(\sigma)=pr^*\omega(\sigma)-(C_0^*)\iota_n(\sigma)
\end{align*}
which equals $(pr^*\omega-C^*\iota_n)(\sigma)$.
This concludes the proof of transitivity.\\



\begin{lemma}\label{lem:cob-trans}
For any $v\in C^{n-1}(X,A)$ there exists a cocycle $E\in Z^n(I\times (X,A))$ such that
$ i_0^*E=0,\hspace{1ex} i_1^*E=\delta v.$
In particular, two differential cocycles $[c,\omega,h]$ and $[c,\omega,h']$ are equal
if $h$ and $h'$ differ by a coboundary.
\end{lemma}
\begin{proof}
Suppose $A=\emptyset$.
The proof is based on the Alexander-Whitney map (cf. \cite{tomDieck}, p.240)
$$A: C_n(X\times Y) \rightarrow \bigoplus\limits_{p+q=n} C_i(X)\otimes C_j(Y),\;
A\sigma=\sum_{p+q=n} {}_p(pr_1\circ \sigma)\otimes (pr_2 \circ \sigma)_q$$
which is a natural chain equivalence, where ${}_p\sigma(t_0,\ldots, t_p)
=\sigma(t_0,\ldots, t_p,0\ldots,0)$ denotes the front $p$-face, and similarly
$\sigma_q$ is the back $q$-face, with the zeros up front. Define $E$ as
$$E: C_n(I\times X)\overset{pr\circ A}{\longrightarrow} C_0(I)\otimes C_n(X)
\overset{id\otimes\partial}{\longrightarrow}
	C_0(I)\otimes C_{n-1}(X)\overset{\varepsilon\otimes v}{\longrightarrow} \bbR$$
where $\varepsilon\left( \sum\limits_{x\in I}r_x[\Delta^0\rightarrow \{x\}\subset I]
\right):=\sum\limits_{x\in I} r_x\cdot x$. Using the explicit formula for $A$ it is easy to check $i_0^*E=0, i_1^*E=\delta v$.
In the relative case, by naturality of $A$, $E$ is zero on $C_n(I\times A)$. 
There is an analogous result for smooth
cochains since $A$ takes the complex of smooth cochains to itself . A version with coefficients in a graded vector space is also easily deduced from our result.
\end{proof}
\noindent
Define
\begin{align*}
I: \hat{E}^n(M,N)\rightarrow E^n(M,N),&\hspace{0.5cm} [c,\omega,h]\mapsto [c]\\
R: \hat{E}^n(M,N)\rightarrow \Omega_{cl}^n(M,N;V),&\hspace{0.5cm} [c,\omega,h]\mapsto \omega\\
a: \Omega^{n-1}(M,N;V)\rightarrow \hat{E}^n(M,N),&\hspace{0.5cm} \Theta \mapsto [const_{pt\in E_n}, d\Theta, \Theta]
\end{align*}
The map $a$ is well-defined since $\delta \Theta=d\Theta$ and  $const^*\iota_n=0$. Also, by Lemma \ref{lem:cob-trans}, $a$ is zero on $im(d)$. It is clear also that
$ch\circ I=can\circ R$ and $R\circ a = d$.

\begin{theorem}\label{thm:axiom-verific}
The following sequence is exact:
$$ E^{n-1}(M,N)\overset{ch}{\longrightarrow} \Omega^{n-1}(M,N;V)/im(d) \overset{a}{\longrightarrow} \hat{E}^n(M,N) \overset{I}{\longrightarrow} E^{n}(M,N) \longrightarrow 0 $$
\end{theorem}
\begin{proof}
We begin with two simple observations from the relative de Rham isomorphism:
\begin{enumerate}
\item[(i)] Every smooth cocycle is cohomologous to a closed form.
\item[(ii)] If a closed form bounds a singular cochain it also bounds a form.
\end{enumerate}
The map $I$ is surjective by (i) applied to the cocycle $c^*\iota_n$. It is also clear that $I\circ a=0$.\\

Suppose that $I[c,\omega,h]=0$, i.e. that we have a homotopy $C: c\simeq const$ rel $N$. Since
also $C\simeq const$ rel $N$ we may pick $e\in C^{n-1}(I\times (M,N); V)$ with $\delta e=C^*\iota_n$.
Consider $$u=h+i_0^*e-i_1^*e\in C^{n-1}_{s}(M,N;V)$$ for the inclusions $i_0, i_1: (M,N)\rightarrow I\times (M,N)$. Then $\delta u=\omega$ is a differential form so that by (ii) we may write $\delta u = d\kappa$ for a differential form $\kappa\in \Omega^{n-1}(M,N;V)$. By (i) we may then write $u-\kappa=\eta+\delta v$ for $\eta\in \Omega^{n-1}_{cl}(M,N;V), v\in C^{n-2}_{s}(M,N;V)$.
By Lemma \ref{lem:cob-trans} choose a cocycle $E\in Z^{n-1}_{s}( I\times (M,N);V)$ with
$ i_0^*E=0,\hspace{1ex} i_1^*E=\delta v$.
Setting $H:=pr^*h+(pr^*i_0^*-id)e-E \in C^{n-1}_{s}(I\times(M,N);V)$ and $\Theta=\kappa+\eta$, the
pair $(C,H)$ witnesses $(c,\omega,h)\sim (const, d\Theta, \Theta)=a(\Theta)$.

Suppose next that $(const,0,0)\sim a(\Theta)=(const,d\Theta,\Theta)$ i.e. that we have $d\Theta=0$
and are given $H\in C^{n-1}_{s}(I\times(M,N);V)$, $C: const\simeq const: (M,N)\rightarrow (E_n,pt)$ with
$$ i_0^*H=\Theta,\; i_1^*H=0,\hspace{1cm} \delta H=-C^*\iota_n.$$
We have to show that $\Theta$ is cohomologous to an element of type $c^*\iota_{n-1}$ for a map
$c: (M,N)\rightarrow (E_{n-1},pt)$. Take $c=(\varepsilon_{n-1}^{adj})^{-1}\circ C^{adj}$, i.e. $C=\varepsilon_{n-1}\circ(\Sigma c)$. Then
$$ c^*\iota_{n-1}=c^*\int_I \varepsilon_{n-1}^*\iota_n=\int_I C^*\iota_n=-\int_I \delta H=\delta\int_I H + i_0^*H-i_1^*H \equiv \Theta.$$
It follows that $\Theta\equiv ch[c]$. From the following lemma we conclude finally that $a\circ ch=0$, applying it to the homotopy $C: const\simeq const$, corresponding as above to $c: (M,N)\rightarrow (E_{n-1},pt)$.
\end{proof}

\begin{proposition}\label{prop:homotopy}
Given a homotopy $C: c_0\simeq c_1\; (\text{rel } N)$ we obtain for all $\omega,h_0$ equivalences
$$ (c_0,\omega, h_0) \sim (c_1, \omega, h_0 - \int_I C^*\iota_n)\hspace{0.5cm}\text{ in }\hat{E}^n(M,N)$$
\end{proposition}

\begin{proof}
Set $h_1=h_0-\int_I C^*\iota_n$.
We first remark that $\left(c_1, \omega, h_1\right)$ satisfies $\delta\left( h_0-\int_I C^*\iota_n \right)=(\omega - c_0^*\iota_n)+i_0^*C^*\iota_n - i_1^*C^*\iota_n=\omega-c_1^*\iota_n$.
We seek a suitable $H\in C^{n-1}_{s}(I\times (M,N))$ with $i_0^* H=h_0$
, $i_1^* H=h_1$ and $\delta H=pr^*\omega-C^*\iota_n=\delta pr^*h_1+\left((c_1\circ pr)^*-C^*\right)\iota_n $.
The homotopy $K: C\simeq (c_1\circ pr)$
$$ K: I\times I \times (M,N)\rightarrow (E_n,pt),\; (s,t,x)\mapsto \begin{cases}
C(t,x)\;\;(s\leq t),\\
C(s,x)\;\;(s\geq t).
\end{cases}$$
is relative to $I\times N$ and
yields a chain homotopy $$\delta \int_I K^*\iota_n = (i_1)^*K^*\iota_n-(i_0)^*K^*\iota_n=(c_1\circ pr)^*\iota_n-C^*\iota_n$$
We may thus take $H=pr^*h_1+\int_I K^*\iota_n$. Then
\begin{align*}
(i_0)^*H&=h_1+\int_I (id_I\times i_0)^*K^*\iota_n=h_1+\int_I C^*\iota_n=h_0\\
(i_1)^*H&=h_1+\int_I (id_I\times i_1)^*K^*\iota_n=h_1+\int_I (c_1\circ pr)^*\iota_n
\end{align*}
The assertion now follows from Lemma \ref{lem:cob-trans} since $\int_I (c_1\circ pr)^*\iota_n=\int_I pr^*(\omega-\delta h_1)=\int_I pr^*\omega-\int_I pr^*\delta h_1=-\int_I pr^*\delta h_1=\delta \int_I pr^*h_1$ is a coboundary, using the fact that $\int_I pr^*\omega=0$ for differential forms.
\end{proof}

\subsection{The Exact Sequence of a Pair}

\begin{definition}
The associated \emph{flat} theory is defined as
$$ \hat{E}_{flat}^*(M,N):=\ker\left(	\hat{E}^*(M,N)	\overset{R}{\longrightarrow} \Omega^*_{cl}(M,N)	\right) $$
\end{definition}

It has been shown in \cite{bunke_schick_uniqueness} and \cite{hopkins_singer} that $\hat{E}_{flat}^*$ is naturally
isomorphic to $E\bbR/\bbZ^{*-1}$. In particular, it is a generalized cohomology theory and we have a long exact sequence of pairs.

\begin{theorem}\label{thm:exact-seq-pairs}
For any closed submanifold $N\subset M$ we have a natural exact sequence
$$
\xymatrix@R=0.5cm{
\ar@{-}[r]&\hat{E}^{n-1}_{flat}(M,N)\ar[r]^{i^*}	&	\hat{E}^{n-1}_{flat}(M)\ar[r]^{j^*}	&	\hat{E}^{n-1}_{flat}(N)\ar@{->}`r[rd]`[rd]^{\delta_1}`[llldd]`[lldd][lldd]	&\\
& & && \\
& \hat{E}^n(M,N)\ar[r]^{i^*}	&\hat{E}^n(M)\ar[r]^{j^*}	&  \hat{E}^n(N)\ar@{->}`r[rd]`[rd]^{\delta_2}`[llldd]`[lldd][lldd]&\\
& & && \\
& E^{n+1}(M,N)\ar[r]^{i^*}		&E^{n+1}(M)\ar[r]^{j^*}	&	E^{n+1}(N)\ar@{-}[r]&
 }
 $$
Here, the coboundary maps are defined as compositions
$\delta_1: \hat{E}^{n-1}_{flat}(N) \overset{\delta}{\longrightarrow} \hat{E}^n_{flat}(M,N)\overset{\subset}{\longrightarrow} \hat{E}^n(M,N)$ and
$\delta_2: \hat{E}^n(N)\overset{I}{\longrightarrow} E^n(N)\overset{\delta}{\longrightarrow} E^{n+1}(M,N)$.
\end{theorem}
\begin{proof}
The exactness at $\hat{E}^n(M,N)$, $\hat{E}^{n-1}_{flat}(N)$, and at $E^{n+1}(M,N)$ is easy to check from the exact sequences of pairs of $\hat{E}^*_{flat}$ and $E^*$. It is also straightforward to see that the composition of two successive maps in the sequence is zero.\\

Exactness at $\hat{E}^n(M)$:
If $(c,\omega,h)|_N\sim (const,0,0)$ we have $j^*\omega=0$ and
there exist $C: I\times N \rightarrow E_n: c\simeq const$ and $H\in C^{n-1}_{s}(I\times N;V)$ with $i_0^*H=h|_N, i_1^*H=0$ and $\delta H=-C^*\iota_n$. As above, we may assume that $C|_{[0,1/4]\times N}=c|_N\circ pr$ and $H|_{[0,1/4]\times N}=pr^*h|_N$. By Lemma \ref{lem:gluing} applied to $0\times M\cup [0,\frac14[\times N$, $]0,1]\times N$ and $pr^*h$ for $pr: 0\times M\cup [0,\frac14[\times N\rightarrow M$ and $H$ we may find $\tilde{H}\in C^{n-1}_{s}(0\times M\cup I\times N)$ with
$ \tilde{H}|_{0\times M}=h, \; \tilde{H}|_{I\times N}=H $.
Because $0\times N\subset 0\times M\cup I\times N$ is closed we may define $\tilde{C}$ by gluing $C$ and $c$. As above, using Remark \ref{rem:subdivision}, it may be verified that $\delta\tilde{H}=pr^*\omega-\tilde{C}^*\iota_n$.
Since $N\subset M$ is a cofibration we may find a smooth retraction\footnote{Start with a continuous retraction and deform to a smooth map relative to the closed subset $0\times M\cup I\times N$.}
$$ r: I\times M \rightarrow (0\times M)\cup (I\times N) $$
Pulling back along $r$ we obtain $\delta r^*\tilde{H}=pr^*\omega - (\tilde{C}\circ r)^*\iota_n$ and thus an equivalence in $\hat{E}^n(M)$ from $(c,\omega,h)$ to $((\tilde{C}\circ r)|_{1\times M},\omega, r^*\tilde{H}|_{1\times M})$ which may be viewed as an element of $\hat{E}^n(M,N)$.\\

Exactness at $\hat{E}^n(N)$:
Suppose $0=\delta_2(\hat{x})=\delta(I(\hat{x}))$. Then $I(\hat{x})$ has a preimage $y\in E^n(M)$ under $j^*: E^n(M)\rightarrow E^n(N)$. By the surjectivity of $I$ we may write $y=I(\hat{y})$ and then
$ j^*\hat{y}-\hat{x}\in \ker(I)=\mathrm{im}(a) $ so that we may write $a(\Theta)=j^*\hat{y}-\hat{x}$ by Theorem \ref{thm:axiom-verific}. $\Theta\in \Omega^{n-1}(N;V)$ may be extended by (\ref{eqn:KES-forms}) to a differential form $\overline{\Theta}\in \Omega^{n-1}(M;V)$. Then
$$ j^*( \hat{y}-a(\overline{\Theta}) )=j^*\hat{y}-a(\Theta)=\hat{x} $$
\end{proof}

\subsection{Functorial Properties}
A smooth map $f: (M_0,N_0) \rightarrow (M,N)$ induces a \emph{pullback}
	\begin{align}
	f^*=\hat{H}^n(f; (E,\iota, V)): \hat{H}^n((M,N);(E,\iota, V))&\longrightarrow \hat{H}^n((M_0,N_0);(E,\iota, 			V)),\\ [c,\omega,h]&\longmapsto [c\circ f, f^*\omega, f^*h]\notag
	\end{align}
and a homomorphism of graded vector spaces $\mu: V\rightarrow W$ induces a \emph{map of the coefficients}
	\begin{align}
	\mu=\hat{H}^n((M,N); (E, \mu)): \hat{H}^n((M,N);(E,\iota, V))&\longrightarrow \hat{H}^n((M,N);(E,			\mu(\iota), W)),\\ [c,\omega,h]&\longmapsto [c,\mu(\omega),\mu(h)]\notag
	\end{align}
For a pointed map $\phi: E\rightarrow F$ and a cocycle $\iota^F\in Z^n(F,pt; V)$ we have a \emph{transfer} map
	\begin{align}
	\phi_*=\hat{H}^n((M,N); (\phi, V)): \hat{H}^n((M,N);(E,\phi^*\iota^F, V))&\longrightarrow \hat{H}^n((M,N);(F,			\iota^F, V)),\label{map:transfer}\\ [c,\omega,h]&\longmapsto [\phi\circ c, \omega, h]\notag
	\end{align}
Also, for every reduced cochain $\Theta\in C^{n-1}(E,pt;V)$ with $\delta\Theta=\iota-\iota'$ we have an isomorphism
	\begin{align}
\hat{H}^n(\Theta): \hat{H}^n((M,N);(E,\iota,V))&\overset{\cong}{\longrightarrow} \hat{H}^n((M,N);(E,\iota',V)),\label{map:change-of-cocycle}\\
	[c,\omega,h]&\longmapsto [c,\omega,h+c^*\Theta]\notag
	\end{align}
	By Lemma \ref{lem:cob-trans} this isomorphism depends only on the coset $\Theta+B^{n-1}(E,pt;V)$. Notice that if $N=\emptyset$, then it is not necessary that $\Theta$ be a \emph{reduced} cocycle. It is easily checked that these maps are in fact well-defined.\\

Combining (\ref{map:transfer}) and (\ref{map:change-of-cocycle}), we define for a pointed map $\phi: E\rightarrow F$ together with a reduced cochain $\Theta\in C^{n-1}(E,pt;V)$ with $\delta\Theta=\iota^E-\phi^*\iota^F$ the map $(\phi;\Theta)$ as the composition $\phi_*\circ \hat{H}^n(\Theta)$. Explicitly, $(\phi; \Theta)[c,\omega,h]=[\phi\circ c, \omega, h+c^*\Theta]$.
Notice the \emph{composition rule}
\begin{equation}
(\phi; \Theta)\circ (\psi; \kappa)=(\phi\circ\psi; \kappa + \psi^*\Theta)
\end{equation}
Moreover, if $\Phi: \phi_0\simeq \phi_1$ is a homotopy we have by Proposition \ref{prop:homotopy}
\begin{equation}
(\phi_0; \Theta)=\left(\phi_1; \Theta-\int_I \Phi^*\iota^F\right)
\end{equation}

\subsection{The Canonical Maps $\lambda$, $\chi$}\label{ssec:canonical-maps}
First, note that we have a natural map $\lambda$ given by
	\begin{align*}
		\hat{H}^n((M,N); (E,\iota^E, V)) \times \hat{H}^n((M,N); (F,\iota^F, W))&\rightarrow \hat{H}^n((M,N); 				(E\times F; pr_1^* \iota^E\oplus pr_2^*\iota^F, V\oplus W))\\
		([c_0,\omega_0,h_0], [c_1,\omega_1,h_1])&\mapsto [(c_0,c_1), \omega_0\oplus\omega_1,    
h_0\oplus h_1]
	\end{align*}
$\lambda$ clearly is associative:
$\lambda\circ(\lambda\times id)=\lambda\circ(id\times\lambda)$.
Also
\begin{equation}
\lambda\circ\left((\phi;\Theta)\times(\psi; \kappa)\right)=(\phi\times \psi; pr_1^*\Theta+pr_2^*\kappa )\circ\lambda
\end{equation}

Next, using the Acyclic Models Theorem, pick a natural chain homotopy $\delta B(\omega_0\otimes \omega_1)+Bd(\omega_0\otimes \omega_1)=\omega_0\wedge \omega_1 - \omega_0\cup \omega_1$,
unique up to natural chain homotopy. Define a natural map
	\begin{align*}
		\chi: \hat{H}^n(M; (E,\iota^E, V))\times \hat{H}^m(M; (F,\iota^F, W))&\rightarrow
		\hat{H}^{n+m}(M; (E\times F, \iota^E \times \iota^F, V\otimes W))\\
		\left([c_0,\omega_0, h_0], [c_1,\omega_1,h_1]\right)\mapsto &[(c_0,c_1), \omega_0\wedge \omega_1,h]
	\end{align*}
where $h=B(\omega_0\otimes\omega_1)+h_0\cup\omega_1 + (-1)^{|\omega_0|}\omega_0\cup h_1 - h_0\cup \delta h_1$.
Notice that by Lemma \ref{lem:cob-trans} we could have replaced $h_0\cup \delta h_1$ with $(-1)^{|\omega_0|} \delta h_0 \cup h_1$, or half of both, since these differ only by a coboundary.
Since any two choices of $B$ differ by a coboundary, $\chi$ is independent of it, by
Lemma \ref{lem:cob-trans}. Moreover, we have associativity
\begin{equation}
\chi\circ (\chi\times id)=\chi\circ (id\times \chi)\label{eqn:assoc-chi}
\end{equation}
For the proof of (\ref{eqn:assoc-chi}),
suppose $\hat{x}=[c_0,\omega_0,h_0]\in \hat{H}^l(M; (E,\iota^E,U))$, $\hat{y}=[c_1,\omega_1,h_1]\in \hat{H}^m(M; (F,\iota^F,V))$, and $\hat{z}=[c_2,\omega_2,h_2]\in \hat{H}^n(M; (G,\iota^G,W))$. Then $\chi(\hat{x},\chi(\hat{y},\hat{z}))$ equals
\begin{align*}
[&(c_0,c_1,c_2), \omega_0\wedge (\omega_1\wedge \omega_2),
B (\omega_0\otimes(\omega_1\wedge \omega_2))+(-1)^{|\omega_0|}\omega_0\cup B(\omega_1\otimes \omega_2)\\
&+h_0\cup \omega_1\cup \omega_2+(-1)^{|\omega_0|} \omega_0\cup h_1\cup \omega_2 + (-1)^{|\omega_0|+|\omega_1|}\omega_0\cup\omega_1\cup h_2\\
&-h_0\cup \delta h_1\cup \omega_2 - h_0\cup \omega_1 \cup \delta h_2 - (-1)^{|\omega_0|}\omega_0\cup h_1\cup \delta h_2+h_0\cup\delta h_1\cup \delta h_2]
\end{align*}
while $\chi(\chi(\hat{x},\hat{y}),\hat{z})$ equals, using due to the remark above $(-1)^{|\omega_0|} \delta h_0\cup h$ instead of $h_0\cup \delta h$,
\begin{align*}
[&(c_0,c_1,c_2), (\omega_0\wedge \omega_1)\wedge\omega_2,
B((\omega_0\wedge \omega_1)\otimes \omega_2)+B(\omega_0\otimes\omega_1)\cup\omega_2\\
&+h_0\cup\omega_1\cup \omega_2+(-1)^{|\omega_0|}\omega_0\cup h_1\cup \omega_2+(-1)^{|\omega_0|+|\omega_1|}\omega_0\cup \omega_1 \cup h_2\\
&-h_0\cup\delta h_1\cup \omega_2-(-1)^{|\omega_0|+|\omega_1|}\delta h_0\cup \omega_1\cup h_2
-(-1)^{|\omega_0|+|\omega_1|}\omega_0\cup \delta h_1\cup h_2+(-1)^{|\omega_0|+|\omega_1|}\delta h_0\cup \delta h_1\cup h_2]
\end{align*}
The two chain homotopies $B((\omega_0\wedge\omega_1)\otimes\omega_2)+B(\omega_0\otimes\omega_1)\cup \omega_2$ and
$B(\omega_0\otimes(\omega_1\wedge\omega_2))+(-1)^{|\omega_0|}\omega_0\cup B(\omega_1\otimes \omega_2)$ are chain homotopic by Acyclic Models. On closed forms $\omega_0,\omega_1,\omega_2$ they therefore differ only by a coboundary $\delta E$.
Lemma \ref{lem:cob-trans} completes the proof since
the third components of the above two elements differ therefore only by the coboundary of
$$ (-1)^{|\omega_0|+|\omega_1|}h_0\cup \omega_1\cup h_2 + h_0\cup \delta h_1\cup h_2+(-1)^{|\omega_1|}\omega_0\cup h_1 \cup h_2 + E$$
Note also that if $\phi: E\rightarrow \tilde{E},\; \psi: F\rightarrow\tilde{F}$ 
and $\delta\Theta=\iota^E-\phi^*\iota^{\tilde{E}}$, $\delta \kappa=\iota^F-\psi^*\iota^{\tilde{F}}$ we have
\begin{equation}\label{eqn:comp-chi-maps}
\chi\circ \left( (\phi; \Theta)\times(\psi; \kappa) \right) = (\phi\times\psi; \Theta\times\delta\kappa + \Theta\times\iota^F+\iota^E\times\kappa)\circ \chi
\end{equation}

\subsection{Abelian Group Structure}\label{ssec:abelian-group}


Pick maps $\alpha_n: E_n\times E_n\rightarrow E_n$ representing addition
and cochains $A_n\in C^{n-1}(E_n\times E_n;V)$ with $\delta A_n= pr_1^*\iota_n + pr_2^*\iota_n - \alpha_n^*\iota_n$ (possible since the cohomology Chern character is additive).
We would like to define addition as the following composition
\[\xymatrix{
\hat{H}^n((M,N); (E_n,\iota_n, V)) \times \hat{H}^n((M,N); (E_n,\iota_n, V))\ar[d]^{\alpha\circ\lambda}\\
\hat{H}^n((M,N); (E_n\times E_n,pr_1^*\iota_n+ pr_2^*\iota_n, V))\ar[d]^{\hat{H}^n(A_n)}\\
\hat{H}^n((M,N); (E_n\times E_n, \alpha_n^*\iota_n, V))\ar[d]^-{(\alpha_n)_*}\\
\hat{H}^n((M,N); (E_n, \iota_n,V))
}\]
using the maps from the previous section. Explicitly,
\[
[c_0,\omega_0, h_0]+[c_1,\omega_1,h_1]= [\alpha_n(c_0,c_1), \omega_0+\omega_1, h_0+h_1+(c_0,c_1)^*A_n]
\]
The natural transformations $a$, $R$, and $I$ are then group homomorphisms.
However, in general, this operation won't be associative.
From the $E_\infty$-structure of the infinite loop space $E_n$
it is possible however
to make \emph{canonical} choices for the $\alpha_n$ and $A_n$
that will satisfy additional \emph{coherence properties} like that
\begin{equation}\label{coherence:assoc}
pr_{12}^*A_n+(\alpha_n\times id)^*A_n-\int_I H_n^*\iota_n - pr_{23}^*A_n-(id\times \alpha_n)^*A_n
\end{equation}
is a coboundary
for a certain canonical homotopy $H_n: \alpha_n(\alpha_n\times id)\simeq \alpha_n(id\times\alpha_n)$
which means associativity, by Proposition \ref{prop:homotopy}.
In Remark \ref{rem:general-case} we show how to make these canonical choices.
Since the verification of the coherence properties are lengthy, since the general case has already been dealt with in \cite{hopkins_singer}, and since this paper is concerned mainly with the rationally even case anyway, we allow ourselves the following simplification:

\begin{assumption}\label{as:rat-even}
From here on, $E$ is rationally even, i.e. $\pi_i(E)\otimes\bbQ=0$, for all odd $i$.
\end{assumption}

Observe that then for any graded vector space $V$ which is zero in odd degrees we have
 $H^*(E_n;V)=0,\; H^*(E_n\times E_m;V)$ for $n,m$ even and $*$ odd, and similarly for higher products
(use for instance Lemma 3.8 in \cite{bunke_schick_uniqueness}). For $n$ even, let $\alpha_n: E_n\times E_n\rightarrow E_n$ be \emph{any} map representing addition
$$ [f]+[g]=[\alpha_n\circ(f,g)],\hspace{1cm} f,g:X\rightarrow E_n$$
Let $\varphi_n: \Sigma(E_n\times E_n) \rightarrow (\Sigma E_n) \times (\Sigma E_n) \xrightarrow{\varepsilon_n \times \varepsilon_n} E_{n+1}\times E_{n+1}$ be the structure map of the spectrum $E_n\times E_n$, basepointed by the pair of basepoints.
To not burden the notation any further, canonical maps like
$
(\Sigma X)\times Y \longrightarrow  \Sigma(X\times Y),
(\Omega X)\times Y \longrightarrow \Omega(X\times Y),
\Omega(X\times Y)	\overset{\approx}{\longrightarrow} \Omega(X)\times \Omega(Y)
$ will be left out in the following formulas.
Since we have an $\Omega$-spectrum, we may define $\alpha_{n-1}$ by requiring
\begin{equation}\label{eqn:req-3}
\varepsilon_{n-1}\circ \left(\Sigma\alpha_{n-1}\right)=\alpha_n\circ \varphi_{n-1}
\end{equation}
Since suspension is linear, $\alpha_{n-1}$ again represents addition.
For all even $n$, write
\begin{equation}\label{eqn:choice-of-An}
\delta A_n = (pr_1^*\iota_n+pr_2^*\iota_n)-\alpha_n^*\iota_n
\end{equation}
for \emph{some} reduced cochain $A_n\in \tilde{C}^{n-1}(E_n\times E_n;V)$. By Assumption \ref{as:rat-even}, any two
choices differ by a coboundary. We \emph{extend} the $A_n$ to odd indices by setting\begin{equation}\label{choice-of-An-1}
A_{n-1} := -\int_{S^1} \varphi_{n-1}^*A_n\in \tilde{C}^{n-2}(E_{n-1}\times E_{n-1};V)  \hspace{1cm}\text{($n$ even)}
\end{equation}
By (\ref{eqn:coboundary-integral-formula}) and (\ref{eqn:choice-fund-cocycle}), equation (\ref{eqn:choice-of-An}) holds then also for $n-1$.

\begin{theorem}
The binary operation $\alpha$ endows $\hat{H}^n((M,N); (E_n,\iota_n, V))$ with the structure of an Abelian group with zero element $[const,0,0]$.
\end{theorem}
\begin{proof}
We only consider negation. For $n$ even, write $\nu_n: E_n \rightarrow E_n$ for a representative map
of the negation in $E^n$. Write also $\delta N_n=-\iota_n-\nu_n^*\iota_n$ for $N_n\in \tilde{C}^{n-1}(E_n;V)$ which is possible since $ch$ preserves negation. Extend to odd indices as above by requiring 
\begin{equation}\label{eqn:req-1}
\varepsilon_{n-1}\circ \left(\Sigma \nu_{n-1}\right)=\nu_n\circ \varepsilon_{n-1}
\end{equation}
and by
setting $N_{n-1}=-\int_{S^1} \varepsilon_{n-1}^*N_n\in \tilde{C}^{n-2}(E_n;V)$. Then again $ \delta N_{n-1}=-\iota_{n-1}-\nu_{n-1}^*\iota_{n-1}$
and we may define the \emph{negation} $\nu$ in differential cohomology as $ \nu([c,\omega,h])=[\nu_n\circ c, -\omega, -h+c^*N_n]$.
Let us show that $\alpha(\hat{x},\nu(\hat{x}))=0$ for any $\hat{x}=[c,\omega,h]$: For $n$ even let
$H_n: \alpha_n\circ (id,\nu_n)\simeq const$ be a homotopy and extend by requiring
\begin{equation}\label{eqn:req-2}
\varepsilon_{n-1}\circ \left(\Sigma H_{n-1}\right)=H_n\circ (id_I\times \varepsilon_{n-1})
\end{equation}
Now, we compute
\begin{align*}
 \alpha(\hat{x},\nu(\hat{x}))&=\alpha([c,\omega,h],[\nu_n\circ c, -\omega, -h+c^*N_n])
 						=[\alpha_n\circ(c,\nu_n\circ c), 0, c^*N_n+(c,\nu_n\circ c)^*A_n]\\
						&=\left[const, 0, c^*N_n+(c,\nu_n\circ c)^*A_n-\int_I (H_n\circ (c\times id_I))^*\iota_n\right]
\end{align*}
To show that $\alpha(\hat{x},\nu(\hat{x}))=0=[const,0,0]$
it remains by Lemma \ref{lem:cob-trans} to hat that $c^*N_n+(c,\nu_n\circ c)^*A_n-c^*\int_I H_n^*\iota_n \in C^{n-1}((M,N);V)$ is a coboundary. For this it suffices to show that $N_n+(id,\nu_n)^*A_n - \int_I H_n^*\iota_n \in \tilde{C}^{n-1}(E_n;V)$ is a coboundary. For $n$ even it is, by Assumption \ref{as:rat-even}, enough to show that this element is a cocycle:
\begin{align*}
\delta &\left(N_n+(id,\nu_n)^*A_n - \int_I H_n^*\iota_n  \right)=-\nu_n^*\iota_n-\iota_n+(id,\nu_n)^*\delta A_n-\delta\int_I H_n^*\iota_n\\
&=-\nu_n^*\iota_n-\iota_n + (id,\nu_n)^*(pr_1^*\iota_n+pr_2^*\iota_n-\alpha_n^*\iota_n)+ (i_0)^*H_n^*\iota_n-(i_1)^*H_n^*\iota_n\\
&=-\nu_n^*\iota_n-\iota_n - (id,\nu_n)^*(\alpha_n^*\iota_n)+\iota_n+\nu_n^*\iota_n+ (\alpha_n\circ(id,\nu_n))^*\iota_n-const^*\iota_n=0
\end{align*}
Applying $-\int_{S^1} \varepsilon_{n-1}^*: \tilde{C}^{n-1}(E_{n-1};V)\rightarrow \tilde{C}^{n-2}(E_{n-2};V)$ it then follows that
\begin{align*}
\text{coboundary}&=\int_{S^1} \varepsilon_{n-1}^* \left(-N_n-(id,\nu_n)^*A_n + \int_I H_n^*\iota_n\right)\\
			&= N_{n-1}-\int_{S^1} \left(\Sigma(id,\nu_{n-1})\right)^*\varphi_{n-1}^* A_n+\int_{S^1} \int_I (id_I\times  \varepsilon_{n-1})^*H_n^* \iota_n\\
			&= N_{n-1}+(id,\nu_{n-1})^*A_{n-1}-\int_I \int_{S^1} (\Sigma H_{n-1})^*\varepsilon_{n-1}^* \iota_n\\
			&=N_{n-1}+(id,\nu_{n-1})^*A_{n-1}-\int_I H_{n-1}^* \iota_{n-1}
\end{align*}
Changing the order of the integrals has gave a minus sign. The proof
that $[const,0,0]$ is a zero element, of associativity and commutativity uses the same technique, and
will be omited.
\end{proof}

\begin{remark}\label{rem:general-case}
Let $E$ be an $\Omega$-spectrum. We can make \emph{canonical} choices, e.g.
$$ \alpha_n: E_n\times E_n\approx \Omega E_{n+1}\times \Omega E_{n+1} \overset{\kappa}\longrightarrow \Omega  E_{n+1} \approx E_n$$
where the map $\kappa$ is concatenation, which is compatible with the integration of cochains up to explicit coboundary terms.
For $e=\pi: \Delta^2 \rightarrow \Delta^1, (t_0,t_1,t_2)\mapsto (t_0+t_1/2, t_2+t_1/2)$ we have, denoting by $[a,b]$ the $1$-simplex
$ \Delta^1\rightarrow [a,b], (x,y)\mapsto  ya+xb$, the relation $\partial e=[0,\frac12]-[0,1]+[\frac12,1]$
so that
\begin{align*}
\int_{[0,1]}u &= \int_{[0,\frac12]}u+\int_{[\frac12,1]}u-\int_{\partial e}u
		= \int_{[0,\frac12]}u+\int_{[\frac12,1]}u+\left((-1)^{|e|}\delta\int_e u - \int_e \delta u\right)
\end{align*}
This may be used to show that $\alpha_n^*\iota_n$ and $pr_1^*\iota_n+pr_2^*\iota_n$ differ by
the coboundary of the \emph{canonical} cochain
$$ A_n=(\varepsilon_{n}^{adj}\times \varepsilon_{n}^{adj})^*\int_e (\kappa^{adj})^*\iota_{n+1}$$
This observation can be used to exhibit an Abelian group structure in the general case.
One has then to construct higher simplices that exhibit the coherence conditions like (\ref{coherence:assoc}).
\end{remark}

\section{Integration}
In this section we give a more explicit description, on which we will draw later, of the integration map than in \cite{hopkins_singer}. For this we will first define the integration map on pairs. Consider
\begin{align*}
&\int_{S^1}: \Omega^{n+1}(S^1\times M, 1\times M)\longrightarrow \Omega^n(M),\hspace{5ex}
\int_{S^1}: C^{n+1}(S^1\times M, 1\times M)\longrightarrow C^n(M),\\
&\int_{S^1}: E^{n+1}(S^1\times M, 1\times M)\cong \tilde{E}^{n+1}(\Sigma M^+)\cong \tilde{E}^n(M^+)=E^n(M)\\
&\int_{S^1}: (E_{n+1},pt)^{(S^1\times M,1\times M)}\longrightarrow E_n^M,\; c\mapsto \left(\varepsilon_n^{adj}\right)^{-1}\circ c^{adj}
\end{align*}
so that by definition we have
$ \varepsilon_n\circ \Sigma\left(  \int_{S^1}c \right)=c $ for $c \in (E_{n+1},pt)^{(S^1\times M, 1\times M)}$.
Together, they yield a well-defined integration map:
$$ \int_{S^1}: \hat{E}^{n+1}(S^1\times M, 1\times M)\longrightarrow \hat{E}^n(M),\; [c,\omega,h]\mapsto \left[\int_{S^1} c, \int_{S^1} \omega, -\int_{S^1}h\right]$$

%

\begin{proposition}
The integration map on pairs is linear.
\end{proposition}
\begin{proof}
Let $[c_0,\omega_0,h_0], [c_1, \omega_1,h_1] \in \hat{E}^{n+1}(S^1\times M, 1\times M)$.
We have to compare
\begin{align*}
&\left[ \int_{S^1} \alpha_{n+1}(c_0,c_1), \int_{S^1} (\omega_0+\omega_1), -\int_{S^1} \left( h_0+ h_1+(c_0,c_1)^*A_{n+1} \right)  \right]\text{ with}\\
&\left[ \alpha_n\circ \left(\int_{S^1} c_0, \int_{S^1} c_1\right), \int_{S^1}\omega_0+\int_{S^1}\omega_1, -\int_{S^1} h_0-\int_{S^1} h_1+\left(\int_{S^1} c_0,\int_{S^1} c_1\right)^*A_n \right]
\end{align*}
In case $n+1$ is even, by
(\ref{eqn:req-3}) and (\ref{choice-of-An-1}),
 both representing triples are equal.
In case $n+1$ is odd, choose a homotopy
$ H: \alpha_n\circ (\varepsilon_n^{adj}\times \varepsilon_n^{adj})^{-1}\simeq (\varepsilon_n^{adj})^{-1}\circ (\Omega\alpha_{n+1}) $.
Using Proposition \ref{prop:homotopy}, it suffices to show that $ \left(\int c_0, \int c_1  \right)^*A_n- (c_0^{adj}, c_1^{adj})^*\int_{S^1}H^*\iota_n  +\int_{S^1} (c_0,c_1)^* A_{n+1}   $ is a coboundary. But this is the pullback under $(c_0^{adj},c_1^{adj})^*\circ ((\varepsilon_n^{adj}\times \varepsilon_n^{adj})^*)^{-1}=\left( \int c_0, \int c_1 \right)^*$
of 
$$A_n-(\varepsilon_n^{adj}\times \varepsilon_n^{adj})^*\int_{S^1} H^*\iota_n+\int_{S^1} \varphi_n^* A_{n+1}\in C^{n-1}(E_n\times E_n; V)$$
It is routine to verify that this is a cocycle -- and thus a coboundary -- using (\ref{eqn:choice-of-An}) and (\ref{eqn:coboundary-integral-formula}).
\end{proof}

\subsection{Absolute Case}

Let $i: 1\times M\subset S^1\times M$, $j: S^1\times M \subset (S^1\times M, 1\times M)$ and $pr_2: S^1\times M\rightarrow 1\times M$ be the projection. We define an integration map on generalized cohomology:
\begin{equation}\label{eqn:susp}
\int_{S^1}: E^{n+1}(S^1\times M)\twoheadrightarrow E^n(M)
\end{equation}
For $x\in E^{n+1}(S^1\times M)$ write $x-pr_2^*i^*x=j^*y$ for a unique $y\in E^{n+1}(S^1\times M,1\times M)$. Then put
$\int_{S^1}x:=\int_{S^1}y$.
Since $pr_2^* i^*j^* y=0$ for $y\in E^{n+1}(S^1\times M, 1\times M)$ we have a commutative diagram
$$\xymatrix{
E^{n+1}(S^1\times M, 1\times M)\ar[d]_{\int_{S^1}}^\cong\ar[r]^-{j^*}	&	E^n(S^1\times M)\ar@{->>}[dl]^{\int_{S^1}}\\
E^n(M)
}$$
So (\ref{eqn:susp}) is indeed surjective.
Also $\int_{S^1}pr_2^*=0$ and analogously for differential forms.
We extend the integration map on pairs in differential cohomology similarly to a map
$$ \int_{S^1}: \hat{E}^{n+1}(S^1\times M)\longrightarrow \hat{E}^n(M) $$
such that $\int_{S^1}pr_2^*=0$ and $\int_{S^1}j^*=\int_{S^1}$. This goes as follows:
the map
$ \hat{E}^{n+1}(S^1\times M) \xrightarrow{id-pr_2^*i^*} \hat{E}^{n+1}(S^1\times M)$
composes with $i^*$ to zero and by Theorem \ref{thm:exact-seq-pairs} therefore has values in the image of
$ \hat{E}^{n+1}(S^1\times M, 1\times M)\overset{j^*}{\longrightarrow} \hat{E}^{n+1}(S^1\times M)$.
For $\hat{x}\in \hat{E}^{n+1}(S^1\times M)$ let $\hat{y}$ denote \emph{any} preimage of $\hat{x}-pr_2^*i^*\hat{x}$ under $j^*$.\\

\begin{proposition}\label{prop:abs-int-welldef}
The following is well-defined:
$$ \int_{S^1} \hat{x}:=\int_{S^1}\hat{y} $$
Moreover (since $i^*pr^*_2=0$ and $i^*j^*=0$ in differential cohomology), we have
$\int_{S^1}pr_2^*=0$,
$\int_{S^1}j^*=\int_{S^1}$
Also, the integration map on $S^1\times M$ is linear.
\end{proposition}
\begin{proof}
We first remark that integration on pairs anticommutes\footnote{Since we are integrating from the left instead of the right (as in \cite{bunke_schick_uniqueness}).} with $a$ and commutes with $R$ and $I$.
Consider the commutative diagram with exact rows
\begin{align*}
\small\xymatrix{
E^n(S^1\times M)\ar[r]^-{ch}				&\Omega^n(S^1\times M;V)/im(d)\ar[r]^a					&\hat{E}^{n+1}(S^1\times M)	\ar[r]^I				& E^{n+1}(S^1\times M)\\
E^n(S^1\times M,1\times M)\ar[r]^-{ch}\ar[u]^{j^*}		&\Omega^n(S^1\times M,1\times M;V)/im(d)\ar[r]^-a\ar[u]^{j^*}		&\hat{E}^{n+1}(S^1\times M,1\times M)\ar[r]^I\ar[u]^{j^*}	& E^{n+1}(S^1\times M,1\times M)\ar@{_(->}[u]^{j^*}
}
\end{align*}
Suppose $j^*\hat{y}=0$ for $\hat{y}\in \hat{E}^{n+1}(S^1\times M, 1\times M)$. Because we have shown that integration on pairs is linear it suffices to show that $\int_{S^1}\hat{y}=0$.  Since $0=I(j^*(\hat{y}))=j^*I(\hat{y})$ and because
$j^*: E^{n+1}(S^1\times M, 1\times M)\rightarrow E^{n+1}(S^1\times M)$ is injective
we may write $\hat{y}=a(\Theta)$. $j^*\Theta$ lies in the kernel of $a$ and therefore $j^*\Theta = ch(t)$ for some $t\in E^n(S^1\times M)$. But $t-pr_2^*i^*(t)$ lies in the kernel
of $i^*$ so that we may write $t-pr_2^*i^*(t)=j^*s$ for some $s\in E^n(S^1\times M, 1\times M)$. We have
$$ \int_{S^1} s=\int_{S^1}j^*s=\int_{S^1}t-\int_{S^1}pr_2^*i^*t =\int_{S^1} t$$
from which it follows that $\int_{S^1}ch(s)=ch(\int_{S^1} s)=ch(\int_{S^1} t)=\int_{S^1}ch(t)$ and that $\int_{S^1}\hat{y}=0$.
Finally, the integration on $S^1\times M$ is linear since
if $j^*\hat{y}_1=\hat{x}_1-pr_2^*i^*\hat{x}_1$ and $j^*\hat{y}_2=\hat{x}_2-pr_2^*i^*\hat{x}_2$ then we may choose $\hat{y}_1+\hat{y}_2$ in the preimage of
$\hat{x}_1+\hat{x}_2-pr_2^*i^*(\hat{x}_1+\hat{x}_2)$.
\end{proof}

\begin{proposition}\label{prop:integration-natural}
Integration is natural with respect to smooth maps $f: M\rightarrow N$:
\begin{align}
f^*\int_{S^1}\hat{x}=\int_{S^1} (id_{S^1}\times f)^*\hat{x}\hspace{1cm}\forall \hat{x}\in\hat{E}^{n+1}(S^1\times M)
\end{align}
and similarly for pairs: if $\hat{x}\in \hat{E}^{n+1}(S^1\times M,1\times M)$ then
$ f^*\int_{S^1}\hat{x}=\int_{S^1} (id_{(S^1,1)}\times f)^*\hat{x} $.
\end{proposition}
\begin{proof}
This is clear for differential forms, cochains, and also maps:
$(c\circ(id_{S^1}\times f))^{adj}=c^{adj}\circ f$. It therefore follows for integration on pairs. For
integration on $S^1\times M$ it then also follows because $\hat{x}-pr_2^*i^*\hat{x}=j^*\hat{y}$ implies
$  (id\times f)^*\hat{x}-pr_2^*i^*(id\times f)^*\hat{x}=j^*(id\times f)^*\hat{y}$.
\end{proof}

\begin{proposition}\label{prop:integral-RIa-compatible}
Integration commutes with the maps $R,I$ and anticommutes with $a$.
\end{proposition}
\begin{proof}
This is readily verified for integration on pairs. But $\int_{S^1} pr_2^*i^*=0$ for forms and cohomology
classes of $S^1\times M$, from which the assertion follows.
\end{proof}

\subsection{Multiple Integrals}

On $S^1\times S^1\times M$ there is also an \emph{integration over the second variable}, defined by $\int_{S^1}^{'}=\int_{S^1} \tau^*$ for the twist $\tau: S^1\times S^1\times M\rightarrow S^1\times S^1\times M,\, (z,w,m)\mapsto (w,z,m)$. 

\begin{proposition}\label{prop:mult-integral}
$ \int_{S^1}\int_{S^1} \tau^*=-\int_{S^1}\int_{S^1}: \hat{E}^{n+2}(S^1\times S^1\times M)\longrightarrow  \hat{E}^n(M)$ for all even $n$:
\begin{equation}\label{eqn:mult-integral}
\int_{S^1} \int_{S^1}^{'}=-\int_{S^1}\int_{S^1}
\end{equation}
\end{proposition}

Note first that the corresponding equation (\ref{eqn:mult-integral}) is true for cochains (since $\tau^*B([S^1]\otimes [S^1])=(-1)^{1\cdot 1} B([S^1]\otimes [S^1])$ where $B$ denotes the Alexander-Whitney map) and also for differential forms. It is also true for functions $c:S^1\times S^1\times M \rightarrow E_{n+2}$ in the sense that
$ \int_{S^1}\int_{S^1} \tau^*c = twist\circ \int_{S^1}\int_{S^1} c$
where $E_n\approx \Omega^2 E_{n+2}\xrightarrow{twist} \Omega^2 E_{n+2}\approx E_n$ interchanges the two path variables.
The following lemma allows one integrate twice
``directly'' on suitable cocycle representatives:


\begin{lemma}\label{lem:techn-int}
Let $\hat{x}\in \hat{E}^{n+2}(S^1\times S^1\times M, 1\times S^1\times M)$ and denote
\begin{align*}
i_{13}: S^1\times 1\times M \subset S^1\times S^1\times M,&\hspace{2ex} i_{23}: 1\times S^1\times M \subset S^1\times S^1\times M\\
i'_{13}: (S^1\times 1\times M,1\times 1\times M) &\subset (S^1\times S^1\times M, 1\times S^1\times M)\\
pr'_{13}: (S^1\times S^1\times M, 1\times S^1\times M) &\rightarrow (S^1\times 1 \times M, 1\times 1\times M)
\end{align*}
Then $\hat{y}=\hat{x}-{pr'}_{13}^*{i'}_{13}^*\hat{x}$ may be represented by a triple $(c,\omega,h)$ with
\begin{align*}
&c|_{(S^1\vee S^1)\times M}=const_{\,pt \in E_{n+2}},\hspace{2ex}
i_{13}^*\omega=0\text{ and }\; i_{23}^*\omega=0,\hspace{2ex}
h|_{(S^1\vee S^1)\times M}=0
\end{align*}
\end{lemma}
\begin{proof}[Proof of Lemma]
Write $\hat{x}=[c,\omega,h]$.
Replacing $\hat{x}$ by $\hat{\hat{x}}=\hat{x}-{pr'}_{13}^*{i'}_{13}^*\hat{x}$ and noticing
$\hat{\hat{x}}-{pr'}_{13}^*{i'}_{13}^*\hat{\hat{x}}=\hat{\hat{x}}=\hat{x}-{pr'}_{13}^*{i'}_{13}^*\hat{x}$
we may assume without loss of generality that ${i'}_{13}^*\hat{x}=0$ which implies that ${i}_{13}^*\omega=0$ and that we have a homotopy
$c|_{S^1\times 1\times M}\simeq const$ rel $1\times 1\times M$. We extend this homotopy by $const$ and $c$ to a map
$$ I\times (S^1\times 1\times M \cup 1\times S^1\times M)\cup 0\times (S^1\times S^1\times M)  \longrightarrow  E_{n+2}$$
Since $S^1\times 1\times M\cup 1\times S^1\times M \subset S^1\times S^1\times M$ is a closed cofibration
we may extend this map to a homotopy $C$ rel $1\times S^1\times M$ from $c$ to a map $\tilde{c}:=C_1$ with $\tilde{c}|_{S^1\times 1\times M\cup 1\times S^1\times M}=const$. It follows from Propostion \ref{prop:homotopy} that, denoting $\tilde{h}=h-\int_I C^*\iota_{n+2}$,
$$ \hat{x}=\left[ \tilde{c},\omega, h-\int_I C^*\iota_{n+2} \right]=[\tilde{c},\omega,\tilde{h}] $$
We have therefore adjusted $c$ according to the requirement. Then $\hat{y}=\hat{x}-{pr'}_{13}^*{i'}_{13}^*\hat{x}$ is the triple representative we seek.

\end{proof}

\begin{proof}[Proof of Proposition \ref{prop:mult-integral}]
For $\hat{z}\in \hat{E}^{n+2}(S^1\times S^1\times M)$ write $\hat{z}-pr_{23}^*i_{23}^*\hat{z}=j_{23}^* \hat{x}$ for some $\hat{x}\in \hat{E}^(S^1\times S^1\times M, 1\times S^1\times M)$. Then $\hat{y}=\hat{x}-pr_{13}^*i_{13}^*\hat{x}$ may be represented by
a triple $(c,\omega,h)$ as in Lemma \ref{lem:techn-int}. Also,
$$-\int_{S^1}\int_{S^1} \hat{z}=-\int_{S^1}\int_{S^1} \hat{y}=\left[\nu_n\circ \left(\int_{S^1}\int_{S^1} c\right), -\int_{S^1}\int_{S^1}\omega, -\int_{S^1}\int_{S^1}h+ \left(\int_{S^1}\int_{S^1} c\right)^*N_n\right]$$
while $$\int_{S^1}\int_{S^1} \tau^*\hat{z}=\int_{S^1}\int_{S^1} \tau^*\hat{y}=\left[twist\circ\left( \int_{S^1}\int_{S^1}c\right), \int_{S^1}\int_{S^1}\tau^*\omega, \int_{S^1}\int_{S^1}\tau^*h\right]$$
It is sufficient now to observe that $N_n-\int_I H_n^*\iota_n$ is a coboundary for a homotopy $H_n: \nu_n \simeq twist$, by verifying that it is a cocycle.
\end{proof}

For more details on the proof of Proposition \ref{prop:mult-integral} we refer to \cite{meine_promotion}.

\section{Products}

%
Suppose now that $E_n$ is an $\Omega$-spectrum representing a multiplicative cohomology theory.
For a better understanding of some of the points arising in the product construction later, let's 
first assume that the spaces $E_n$ are manifolds and the $\iota_n$ differential forms. $S$ will typically denote a smooth manifold in this section.
In this case we have a natural lift of $can: E_n^S\rightarrow [S, E_n]$ to $\hat{H}^n(S; (E_n,\iota_n); V)$ given by $s(c)=[c,c^*\iota_n,0]$. From this we obtain a natural lift $\tilde{ch}=R\circ s: E_n^S\rightarrow \Omega^n_{cl}(S;V)$ of $ch\circ can$.
Let $$\mu_{n,m}: E_n\wedge E_m \rightarrow E_{n+m}$$ be maps representing the multiplication in the spectrum. We will write also
$$c_1 \cup c_2 := \mu_{n,m}\circ (c_1,c_2),\hspace{1cm} (c_1\in E_n^S, c_2\in E_m^S)$$
Assume now that $\hat{H}^*$ has a multiplicative structure. Then compute
\begin{align*}
\tilde{ch}(c_1)\wedge\tilde{ch}(c_2)&=Rs(c_1)\wedge Rs(c_2)=R(s(c_1)\cup s(c_2))\\
	&= R(s(c_1\cup c_2)+s(c_1)\cup s(c_2)-s(c_1\cup c_2))\\
	&=\tilde{ch}(c_1\cup c_2)+R\left(s(c_1)\cup s(c_2)-s(c_1\cup c_2)\right)
\end{align*}
Since $I(s(c_1)\cup s(c_2)-s(c_1\cup c_2))=0$
we may write 
$$a(\varsigma(c_1,c_2))=s(c_1)\cup s(c_2)-s(c_1\cup c_2)$$
for a natural transformation
$$\varsigma: E_n^S\times E_m^S \Longrightarrow \Omega^{n+m-1}(S)/im(d)/im(ch)$$
which, by Yoneda's Lemma, is implemented by cosets of differential forms $M_{n,m}\in \Omega^{n+m-1}(E_n\times E_m; V)/im(d)$ (uniquely determined up to $im(ch)$) via the formula $\varsigma(c_1,c_2)=(c_1,c_2)^*M_{n,m}$.
Taking $c_1=pr_1, c_2=pr_2$ we see that $M_{n,m}$ is given by the condition
$ a(M_{n,m})=s(id_n)\times s(id_m)-s(\mu_{n,m})$.
Applying $R$ yields as a necessary condition for a product on $\hat{H}^*$:
\begin{equation}
dM_{n,m}=\iota_n\times \iota_m - \mu_{n,m}^*\iota_{n+m}
\end{equation}
Conversely, this is precisely the condition that enables us to define a multiplication map on $\hat{H}^*$, using the isomorphisms $\hat{H}^*(M_{n,m})$ induced by the $M_{n,m}$'s and the transfer map $(\mu_{n,m})_*$
$$ \hat{H}^*(S; (E_n\times E_m, \iota_n\times \iota_m); V)\cong \hat{H}^*(S; (E_n\times E_m; \mu_{n,m}^*\iota_{n+m}); V)\rightarrow \hat{H}^*(S; (E_{n+m},\iota_{n+m}); V)$$
Notice that this map depends only on the coset $(M_{n,m}+im(\delta))+im(ch)$, by Lemma \ref{lem:cob-trans} and Proposition \ref{prop:homotopy}.\\

If the product on $\hat{H}^*$ is associative, bilinear, has a unit, or is graded commutative then
the $M_{n,m}$ will not be completely arbitrary, but satisfy certain \emph{coherence conditions}, e.g.
$$ (id_n\times \mu_{m,l})^*M_{n,m+l}+\iota_n\times M_{m,l}\equiv (\mu_{n,m}\times id_l)^*M_{n+m,l}+M_{n,m}\times \iota_l-\int_I H^*\iota_{n+m+l}  $$
modulo coboundaries and $im(ch)$ for associativity. Here, $H$ denotes a homotopy $H: \mu_{n+m,l}\circ (\mu_{n,m}\times id_l)\simeq \mu_{n,m+l}\circ (id_n\times \mu_{m,l})$. Conversely,
whenever this coherence condition is met, we are able to define an associative product (cf. the proof of associativity in Theorem \ref{thm:products-even-degrees})\footnote{Such a condition might of course be ensured also by working only with ring spectra in a stricter sense.}.

\begin{proof}[Proof of claimed Coherence Condition]
If the product is associative then certainly $(s(id_n)\times s(id_m) )\times s(id_l)$ is equal to
$s(id_n)\times (s(id_m)\times s(id_l))$. The first is
\begin{align*}
(s(id_n)\times s(id_m) )\times s(id_l)&=\left( s(\mu_{n,m})+a(M_{n,m}) \right)\times s(id_l)\\
&=(\mu_{n,m}\times id_l)^* s(id_{n+m})\times s(id_l)+a(M_{n,m}\times Rs(id_l))\\
&=(\mu_{n,m}\times id_l)^* \left( s(\mu_{n+m,l})+a(M_{n+m,l})\right)+a(M_{n,m}\times \iota_l)\\
&=(\mu_{n+m,l}(\mu_{n,m}\times id_l) )^* s(id_{n+m+l})+a\left( (\mu_{n,m}\times id_l)^*M_{n+m,l}+M_{n,m}\times \iota_l \right)
\end{align*}
using $a(\Theta) \times \hat{x}=a(\Theta \times R(\hat{x}))$ from the axioms for products. Similarly, the second is
\begin{align*}
s(id_n)\times (s(id_m)\times s(id_l))&=(\mu_{n,m+l}(id_n\times \mu_{m,l}))^*s(id_{n+m+l})+
a\left( (id_n\times \mu_{m,l})^*M_{n,m+l}+\iota_n\times M_{m,l} \right)
\end{align*}
Now we use the \emph{homotopy formula}, which follows from the axioms for any differential generalized cohomology theory, c.f. \cite{bunke_schick_uniqueness}:
$$ i_1^*\hat{x}-i_0^*\hat{x}=a\left( \int_I R(\hat{x}) \right),\hspace{1cm} \forall\hat{x}\in \hat{E}^{*}(I\times M) $$
Taking $\hat{x}=H^*s(id_{n+m+l})$ here, we conclude that
$$(\mu(id\times\mu))^*s(id)-(\mu(\mu\times id))^*s(id)=a\left( \int_I R(H^*s(id))  \right)=a\left(\int_I H^*\iota_{n+m+l}\right)$$
The assertion now follows from the exact sequence in Theorem \ref{thm:axiom-verific}.
\end{proof}

As already remarked, the product depends on the coset $(M_{n,m}+im(\delta))+im(ch)$.
Different cosets
correspond in general to \emph{different multiplicative structures}: Two multiplications $\cup$ and $\tilde{\cup}$ determine a natural transformation $$K: E^n\times E^m\Longrightarrow H_{dR}^{n+m-1},\hspace{0.5cm} a\left( K(I(x),I(y)) \right)=x\tilde{\cup}y - x\cup y,\hspace{1cm} (x,y\in \hat{H}^n(S;(E_n,\iota_n);V))$$ which
arises by Yoneda's Lemma from elements $\eta_{n,m}\in H_{dR}^{n+m-1}(E_n\times E_m;V)$. We then have $\tilde{M}_{n,m}=M_{n,m}+\eta_{n,m}$. Conversely, choosing $\eta$'s yields a $K$ and thus another bilinear product.

\subsection{Construction in Even Degrees}

For rationally even spectra all these ambiguities disappear for even $n$. If we demand \emph{compatibility with integration} then the product is determined in all degrees.
Thus suppose that $E^*$ is a rationally even multiplicative cohomology theory with units\footnote{Otherwise $\hat{E}$ will of course not have units either. The other arguments go through, though.}, with internal product $\cup$, and associated external product $\times$.\\

Let $\mu_{n,m}: E_n\wedge E_m \rightarrow E_{n+m}$ denote maps representing multiplication. We
use the unqualified $\mu$ to denote the induced multiplication $V\otimes V\rightarrow V$ on the \emph{coefficients} $V^*=E^*(pt)\otimes \bbR$. Choose also a map $e: pt\rightarrow E_0$ representing the unit.\\

Recall that by Acyclic Models we have a (unique up to homotopy) natural chain homotopy with the property that for differential forms $\omega_0, \omega_1$ we have
$$\delta B(\omega_0\otimes \omega_1)+B d(\omega_0\otimes \omega_1)= \omega_0\wedge \omega_1-\omega_0\cup\omega_1$$

Since $ch$ is compatible with the product we have
$0=ch(id)\times ch(id)-ch(id\times id)=\mu(\iota_n\times \iota_n)-\mu_{n,m}^*\iota_{n+m}$ in reduced cohomology. We may thus \emph{choose} reduced cochains $M_{n,m}\in \tilde{C}^{n+m-1}(E_n\wedge E_m; V)$ with
$$\delta M_{n,m}=\mu(\iota_n\times \iota_m)-\mu_{n,m}^*\iota_{n+m},\hspace{1cm}(n,m \text{ even})$$

By assumption (rationally even), any two choices differ by a coboundary. The following \emph{product in even degrees}
is thus \emph{independent} (Lemma \ref{lem:cob-trans}) of the choice of the cochains $M_{n,m}$:
$$\xymatrix{
\hat{H}^n(S; (E_n, \iota_n); V)\times \hat{H}^m(S; (E_m, \iota_m); V)\ar[d]^{\mu\circ \chi}\\
\hat{H}^{n+m}(S; (E_n\times E_m, \mu(\iota_n\times\iota_m); V)\ar[d]^{\hat{H}^{n+m}(M_{n,m})}\\
\hat{H}^{n+m}(S; (E_n\times E_m, \mu_{n,m}^*\iota_{n+m}); V)\ar[d]^{(\mu_{n,m})_*}\\
\hat{H}^{n+m}(S; (E_{n+m}, \iota_{n+m}); V)
}$$
Explicitly, $\mu\left([c_0,\omega_0,h_0],\, [c_1,\omega_1,h_1] \right)$ equals
$$\left[ \mu_{n,m}(c_0,c_1),\omega_0\wedge \omega_1,  B(\omega_0\otimes\omega_1)+h_0\cup\omega_1+(-1)^n\omega_0\cup h_1-h_0\cup\delta h_1 + (c_0,c_1)^*M_{n,m} \right]$$
Here we have suppressed the coefficient map $\mu$ from the notation;
By Lemma \ref{lem:cob-trans} different choices for the chain homotopy $B$ yield
the same map in differential cohomology.

\begin{theorem}\label{thm:products-even-degrees}
The product just defined on $\hat{E}^{2*}(S)$ in even degrees is
(graded) commutative,
bilinear,
associative, and
has a unit $\hat{1}_S$, to be defined in the proof (\ref{eqn:unit}).
In other words, $\hat{E}^{2*}$ is a functor to unital graded commutative rings (cf. Proposition \ref{prop:pullback-mult})
\end{theorem}

\begin{proof}

We first define the \textbf{neutral elements}.
By naturality, the unit $1_S$ in $E^0(S)$ equals $\pi^*1_{pt}$ for the projection $\pi=\pi_S: S\longrightarrow \{pt\}$. Choose a representing map
$ u: \{pt\}\longrightarrow E_0 $
for $1_{pt}$. Then $V^*=E^*(pt)\otimes \bbR$ has unit $1_{pt}\otimes 1$. The singular cohomology $H^*(-;V)$ with coefficients in $V$ of course also has natural units coming from a cohomology class $1_{pt}\in H^0(pt; V)=H^0(pt;V^0)$. It may be represented by a cochain that comes from a differential form (under the deRham map)
$$ \omega_{pt}\in \Omega^0(pt;V)=\Omega^0(pt;V^0),\hspace{0.5cm} \omega_{pt}(pt):=1_{pt}\otimes 1 $$
Because the Chern character preserves units,
$ u^*[\iota_0]=ch([u])=ch(1_{pt})=1_{pt} $, 
we may write
$\delta U = \omega_{pt}-u^*\iota_0$
for some cochain $U\in C^{-1}(pt; V)$.
For a smooth manifold $S$ define the \emph{unit in differential cohomology}
\begin{equation}\label{eqn:unit}
\hat{1}_S=\big[ u\circ \pi_S, \pi_S^*\omega_{pt}, \pi_S^*U  \big]
\end{equation}
Notice that since $E$ is rationally even, the definition of $\hat{1}_S$ does not depend
on the choice of the ``refining cochain'' $U$. Clearly, $\pi_S^*\omega_{pt}$ is just the
unit of $\Omega^*(S;V)$.\\

First note that $B(\pi_S^*\omega_{pt}\otimes-)$ is a chain homotopy from $\eta\mapsto \pi_S^*\omega_{pt}\wedge \eta$ to $\eta\mapsto \pi_S^*\omega_{pt}\cup\eta$ -- just as the zero map is. It follows that $B(\pi_S^*\omega_{pt}\otimes\eta)$ is a coboundary for every cocycle $\eta$. Hence
\begin{align*}
\hat{1}_S&\cup \hat{x}\\=&\big[  \mu_{0,n}(u\pi_S,c),\pi_S^*\omega_{pt}\wedge\omega, B(\pi_S^*\omega_{pt}\otimes\omega) + \pi_S^*U\cup\omega+\pi_S^*\omega_{pt}\cup h -\pi_S^*U\cup\delta h+(u\pi_S,c)^*M_{0,n}  \big]\\
=&\Big[\mu_{0n}(u\pi_{E_n},id)\circ c,\omega, h+c^*\big(\pi_{E_n}^*U\cup\iota_n+(u \pi_{E_n},id)^*M_{0,n}  \big)\Big]
\end{align*}
selecting a homotopy $H: \mu_{0,n}\circ (u\pi_{E_n},id)\simeq id_{E_n}$ and using Proposition \ref{prop:homotopy} we conclude that it suffices to show that
$$ \pi_{E_n}^*U\cup \iota_n + (u\pi_{E_n},id)^*M_{0,n}-\int_I H^*\iota_n \in C^{n-1}(E_n;V) $$
is a coboundary. But it is easily checked to be a cocycle -- which suffices.\\

For \textbf{commutativity}
we first remark that $D(\omega_0\otimes \omega_1):=(-1)^{|\omega_0|\cdot|\omega_1|}B(\omega_1\otimes \omega_0)$ also satisfies $\delta D(\omega_0\otimes\omega_1)+D(d(\omega_0\otimes\omega_1))=\omega_0\wedge\omega_1-\omega_0\cup\omega_1$. Therefore $D$ and $B$ are naturally chain homotopic. Let $[c_0,\omega_0,h_0]\in \hat{E}^n(S)$ and $[c_1,\omega_1,h_1]\in \hat{E}^m(S)$. For $n,m$ even and $\omega_0,\omega_1$ closed it follows then that $B(\omega_0\otimes\omega_1)$ and $B(\omega_1\otimes\omega_0)$ differ by a coboundary.
Also $h_1\cup \delta h_0=\delta h_0\cup h_1=\delta(h_0\cup h_1)+h_0\cup\delta h_1\equiv h_0\cup\delta h_1$. Therefore, selecting a homotopy $H: \mu_{m,n}\circ twist\simeq \mu_{n,m}$,
\begin{align*}
[c_1,\omega_1,h_1]&\cup[c_0,\omega_0,h_0]\\
=\big[&\mu_{m,n}(c_1,c_0),\omega_1\wedge\omega_0,
B(\omega_1\otimes \omega_0)+h_1\cup\omega_0 + \omega_1\cup h_0 - h_1\cup\delta h_0+(c_1,c_0)^*M_{m,n}\big]\\
=\Big[&\mu_{n,m}\circ twist\circ (c_1,c_0),\omega_0\wedge\omega_1,\\
&B(\omega_0\otimes \omega_1)+\omega_0\cup h_1+h_0\cup \omega_1-h_0\cup\delta h_1 +(c_1,c_0)^*M_{m,n} - (c_0,c_1)^*\int_I H^*\iota_{n+m}\Big]
\end{align*}
But $M_{n,m}-twist^*M_{m,n}+\int_I H^*\iota_{n+m}$ is a coboundary since it is a cocycle.\\

For \textbf{associativity} it is clearly enough to prove the commutativity of the hexagon on the next page.
The commutativity of \circled{1} is simply the associativity statement for $\chi$ proven earlier. \circled{2} follows from (\ref{eqn:comp-chi-maps}):
$$ (id_n \times \mu_{m,l}; \iota_n\times M_{m,l})\circ\chi = \chi\circ ((id,0)\times (\mu_{m,l};M_{m,l})  ) $$
and \circled{3} is analogous. For \circled{4} select a homotopy $H: \mu_{n,m+l}\circ (id\times \mu_{m,l})\simeq \mu_{n+m,l}\circ (\mu_{n,m}\times id_l)$ and compare
\begin{align*}
(\mu_{n,m+l};M_{n,m+l})\circ &(id_n\times \mu_{m,l}; \iota_n\times M_{m,l})=(\mu_{n,m+l}\circ(id_n\times \mu_{m,l}); M_{n,m+l}+\mu_{n,m+l}^*(\iota_n\times M_{m,l}))\\
&=\big(\mu_{n+m,l}\circ(\mu_{n,m}\times id_l); M_{n,m+l}+\mu_{n,m+l}^*(\iota_n\times M_{m,l})-\int H^*\iota_{n+m+l}\big)
\end{align*}
with
\begin{align*}
\big(\mu_{n+m,l};M_{n+m,l})\circ &(\mu_{n,m}\times id_l; M_{n,m}\times\iota_l)=(\mu_{n+m,l}\circ(\mu_{n,m}\times id_l); M_{n+m,l}+\mu_{n+m,l}^*(M_{n,m}\times\iota_l)\big)
\end{align*}
But $M_{n,m+l}+\mu_{n,m+l}^*(\iota_n\times M_{m,l})-\int H^*\iota_{n+m+l}-M_{n+m,l}-\mu_{n+m,l}^*(M_{n,m}\times\iota_l)$ is a coboundary which is seen, as usual, by showing that is a cocycle.

\newpage${}$\hspace{13cm}
\begin{rotate}{-90}$$\small\xymatrix@R=2cm{
&	\hat{H}^n(S; E_n,\iota_n)\times \hat{H}^m(S; E_m,\iota_m)\times \hat{H}^l(S; E_l,\iota_l)\ar@{}[dd]^*++{\circled{1}}\ar[ld]_{\chi\times id}\ar[rd]^{id\times \chi}	&\\
\hat{H}^{n+m}(S; E_n\times E_m; \iota_n\times \iota_m)\times \hat{H}^l(S; E_l;\iota_l)\ar[rd]^\chi\ar[d]_{(\mu_{n,m};M_{n,m})\times id} &&
\hat{H}^n(S;E_n,\iota_n)\times \hat{H}^{m+l}(S; E_m\times E_l,\iota_m\times\iota_l)\ar[ld]_\chi\ar[d]^{id\times(\mu_{m,l}; M_{m,l})}\\
\hat{H}^{n+m}(S; E_{n+m},\iota_{n+m})\times \hat{H}^l(S;E_l,\iota_l)\ar[d]_\chi\ar@{}[r]^*++{\circled{2}}&
\hat{H}^{n+m+l}(S;E_n\times E_m\times E_l, \iota_n\times \iota_m\times \iota_l)\ar[ld]_*+++{(\mu_{n,m}\times id_l; M_{n,m}\times \iota_l)}\ar[rd]^*+++{(id_n\times \mu_{m,l}; \iota_n\times M_{m,l})}\ar@{}[dd]^*++{\circled{4}}&
\hat{H}^{n}(S;E_n,\iota_n)\times\hat{H}^{m+l}(S; E_{m+l},\iota_{m+l})\ar[d]^\chi\ar@{}[l]_*++{\circled{3}}\\
\hat{H}^{n+m+l}(S; E_{n+m}\times E_l,\iota_{n+m}\times \iota_l)\ar[dr]_*++{(\mu_{n+m,l},M_{n+m,l})}& &
\hat{H}^{n+m+l}(S; E_n\times E_{m+l},\iota_n\times \iota_{m+l})\ar[dl]^*++{(\mu_{n,m+l},M_{n,m+l})}\\
&\hat{H}^{n+m+l}(S; E_{n+m+l},\iota_{n+m+l})&\\
}$$\end{rotate}\newpage

\textbf{Bilinearity} may be proven similarly, and will be ommited.

\end{proof}

\begin{proposition}\label{prop:pullback-mult}
For a smooth map $f: S_1 \rightarrow S_2$ the induced map
$$ f^*: \hat{E}^n(S_2)\longrightarrow  \hat{E}^n(S_1)$$
is multiplicative in even degrees. Moreover, $f^*$ is unital:
$ f^*(\hat{1}_{S_2})=\hat{1}_{S_1} $.
\end{proposition}
\begin{proof}
Both $f^*\left([c_0,\omega_0,h_0]\cup [c_1,\omega_1,h_1]  \right)$ and
$ f^*[c_0,\omega_0,h_0]\cup f^*[c_1,\omega_1,h_1] $ are equal to
$$ \left[\mu_{n,m}(c_0,c_1)\circ f, f^*\omega_0\wedge f^*\omega_1, f^*(B(\omega_0\otimes\omega_1)+h_0\cup \omega_1+(-1)^n\omega_0\cup h_1 - h_0\cup\delta h_1 ) +f^*(c_0,c_1)^*M_{n,m}   \right] $$
using the fact that $B$ is natural. 
\end{proof}

\begin{proposition}
Let $n,m$ be even and $\Theta\in \Omega^{n-1}(S;V), \hat{x}=[c,\omega,h]\in \hat{E}^m(M)$. Then we have
\begin{equation}\label{eqn:partial-a-mult-compatible}
a(\Theta)\cup\hat{x}=a(\Theta\wedge R(\hat{x}))
\end{equation}
\end{proposition}
\begin{proof}
Select a homotopy $H: \mu_{n,m}\circ (const\times id_m)\simeq const$. Now compute
\begin{align*}
a(\Theta)\cup\hat{x}&=[\mu_{n,m}(const,c),d\Theta\wedge\omega, B(d\Theta\otimes\omega)+\Theta\cup\omega+(-1)^n d\Theta\cup h-\Theta\cup\delta h +(const,c)^*M_{n,m}]\\
&= [\mu_{n,m}(const,c), d\Theta\wedge\omega, \Theta\wedge\omega + (const,c)^*M_{n,m}]
\end{align*}
because $B(d\Theta\otimes\omega)\equiv \Theta\wedge\omega-\Theta\cup\omega$ ($\omega$ is closed) and because $(-1)^n d\Theta\cup h-\Theta\cup\delta h$ is a coboundary. We have to compare with
$$ a(\Theta\wedge R(\hat{x}))=[const, d\Theta\wedge\omega,\Theta\wedge\omega]$$
It thus suffices to show that the following is a coboundary -- by applying $\delta$:
$$ (const\times id)^*M_{n,m}-\int_I H^*\iota_{n+m} \in C^{n+m-1}(E_n\times E_m; V)$$
\end{proof}

\subsection{Extension to Odd Degrees}

The extension to odd degrees relies on the following lemma:

\begin{lemma}\label{lem:extension-odd}
For any $\hat{x}\in \hat{E}^{n-1}(M)$ there exists $\hat{X}\in\hat{E}^n(S^1\times M, 1\times M)$ with $\int_{S^1} \hat{X}=\hat{x}$. For any two choices $\hat{X}, \hat{X}'$ there exists a differential form $\Theta\in \Omega^{n-1}(S^1\times M, 1\times M;V)$ with
$\hat{X}-\hat{X}'=a(\Theta)$ and such that $\int_{S^1}\Theta$ lies in the image of the Chern character.
\end{lemma}
\begin{proof}
\textbf{Existence}. By surjectivity
$\hat{E}^n(S^1\times M, 1\times M)\overset{I}{\twoheadrightarrow} E^n(S^1\times M,1\times M)\overset{\int_{S^1}}{\cong} E^{n-1}(M)$ we may choose $\hat{X}\in\hat{E}^n(S^1\times M,1\times M)$ with $\int_{S^1} I(\hat{X})=I(\hat{x})$. Then $\hat{x}-\int_{S^1}{\hat{X}}$ lies in the kernel of $I$ and
thus equals $a(\Theta)$ for some $\Theta\in \Omega^{n-2}(M;V)$. Pick $\alpha: S^1\rightarrow \bbR$ smooth with $\int_{S^1} \alpha(t)dt=1$ and $\alpha(1)=0$ and let $\tilde{\Theta}=-\alpha(t)dt\wedge pr^*_2\Theta \in \Omega^{n-1}(S^1\times M, 1\times M;V)$. Then $\int_{S^1}\tilde{\Theta}=-\Theta$, and $\hat{X}+a(\tilde{\Theta})$ is the element we seek.

\textbf{Uniqueness}. We have $\int_{S^1} I(\hat{X}-\hat{X}')=0$ so that $I(\hat{X}-\hat{X}')$ lies in
the kernel of $\int_{S^1}: E^n(S^1\times M, 1\times M)\cong E^{n-1}(M)$, which is zero.
Thus, $I(\hat{X}-\hat{X}')=0$ so that $\hat{X}-\hat{X}'=a(\Theta)$ for some $\Theta \in \Omega^{n-1}(S^1\times M, 1\times M;V)$. Since $a\int_{S^1}\Theta=-\int_{S^1} a(\Theta)=-\int_{S^1}(\hat{X}-\hat{X}')=0$, the element $\int_{S^1} \Theta$ lies in the image of the Chern character.
\end{proof}
\noindent
We will usually denote a choice of element $\hat{X}\in \hat{E}^n(S^1\times M)$ with $\int \hat{X}=\hat{x}$ by \emph{upper case}.\\

The internal product $\cup$ in even degrees in differential cohomology is equivalent to an associative, bilinear, and commutative\footnote{Meaning $\hat{x}\times \hat{y}=twist^*(\hat{y}\times\hat{x})$} external product $\times$ in even degrees via the usual formulas
\begin{align*}
\hat{x}\times\hat{y}&=pr_1^*\hat{x}\cup pr_2^*\hat{y}\\
\hat{x}\cup\hat{y}&=\Delta^*(\hat{x}\times\hat{y})
\end{align*}
The exterior wedge product of differential forms is denoted by $\overline{\wedge}$.
Explicitly, $\hat{x}\times\hat{y}$ is
\begin{align*}
[c_1,&\omega_1,h_1]\times[c_2,\omega_2,h_2]\\
&=\big[   \mu_{n,m}(c_1\times c_2),\; \omega_1\overline{\wedge}\omega_2,\; B(\omega_1\otimes\omega_2)+h_1\times \omega_2+(-1)^{|\omega_1|}\omega_1\times h_2 - h_1\times\delta h_2 + (c_1\times c_2)^*M_{n,m}   \big]
\end{align*}

\begin{definition}\label{def:mult-ext} Let $n,m$ be even. Define $\hat{x}\times\hat{y}$ by
\begin{align}
&\int_{S^1}\hat{X}\times\hat{y}		&&\text{for }\hat{x}\in \hat{E}^{n-1}(M), \hat{y}\in \hat{E}^{m}(N),\label{eqn:mult-odd-1}\\
&\int_{S^1}\hat{x}\times \hat{Y}		&&\text{for }\hat{x}\in \hat{E}^n(M),\;\;\,\, \hat{y}\in \hat{E}^{m-1}(N),\label{eqn:mult-odd-2}\\
&\int_{S^1}\hat{X}\times\hat{y}=-\int_{S^1}\hat{x}\times\hat{Y}		&&\text{for }\hat{x}\in \hat{E}^{n-1}(M), \hat{y}\in \hat{E}^{m-1}(N).\label{eqn:mult-odd-3}
\end{align}
where $\hat{X}\in \hat{E}^{|\hat{x}|+1}(S^1\times M),\, \hat{Y}\in \hat{E}^{|\hat{y}|+1}(S^1\times N)$ are absolute classes with $\int_{S^1} \hat{X}=\hat{x}, \int_{S^1} \hat{Y}=\hat{y}$.
\end{definition}

\begin{proposition}
The above is well-defined.
Moreover, in all degrees we have the formula
\begin{equation}\label{eqn:aRtimes}
a(\Theta)\times\hat{y}=a(\Theta\times R(\hat{y})),\hspace{1cm} \forall\Theta\in \Omega^{n-1}(M;V), \hat{y}\in \hat{E}^m(N)
\end{equation}
\end{proposition}
\begin{proof}
For the asserted equality in (\ref{eqn:mult-odd-3}) we calculate
\begin{align*}
\int_{S^1} \hat{X}\times\hat{y} &= \int_{S^1} \int_{S^1}^{'} \hat{X}\times\hat{Y}=-\int_{S^1}\int_{S^1} \hat{X}\times\hat{Y}=-\int_{S^1} \hat{x}\times \hat{Y}
\end{align*}
where the slash indicates that we integrate over the second copy of $S^1$. Note that the middle equality follows from Proposition \ref{prop:mult-integral} since $\hat{X}, \hat{Y}$ are of even degree. Next, let us show that (\ref{eqn:mult-odd-1}) is indeed well-defined:
Thus, suppose $\int \hat{X}=\hat{x}=\int \hat{X}'$ and
write $j^*\hat{Z}=\hat{X}-pr_2^*i^*\hat{X}$ and
$j^*\hat{Z}'=\hat{X}'-pr_2^*i^*\hat{X}'$, using Theorem \ref{thm:exact-seq-pairs}, for relative classes $\hat{Z}, \hat{Z}' \in \hat{E}^n(S^1\times M, 1\times M)$, where
$$j: S^1\times M \subset (S^1\times M, 1\times M),\;\; i: 1\times M\subset S^1\times M,\;\; pr_2: S^1\times M\twoheadrightarrow 1\times M$$
Using the rules in Proposition \ref{prop:abs-int-welldef}, it follows that $\int \hat{Z}=\int \hat{X}=\hat{x}=\int \hat{X}'=\int \hat{Z}'$ so that Lemma \ref{lem:extension-odd} implies that $\hat{Z}$ and $\hat{Z}'$ differ by $a(\Theta)$ with $\int_{S^1} \Theta\in can^{-1}(im(ch))$.
Since $R$ has image contained in $can^{-1}(im(ch))$ and since $ch$ is multiplicative in cohomology, $\int_{S^1}\Theta \overline{\wedge} R(\hat{y})$
also represents an element in the image of the Chern character. Thus,
$$ \int_{S^1} a(\Theta)\times \hat{y}\overset{(\ref{eqn:partial-a-mult-compatible})}{=}\int_{S^1} a(\Theta\overline{\wedge} R(\hat{y}))=-a\left( \int_{S^1} \Theta\overline{\wedge} R(\hat{y})  \right)=0$$
which shows that $\int (\hat{X}-\hat{X}')\times \hat{y}=\int (\hat{Z}-\hat{Z}')\times \hat{y}=0$ so
that (\ref{eqn:mult-odd-1}) is really well-defined.\\

Also, we still have $a(\Theta)\times\hat{y}=a(\Theta\overline{\wedge} R(\hat{y}))$ for $\Theta \in \Omega^{n-2}(M;V)$, $\hat{y}\in \hat{E}^m(N)$ and $n,m$ even: Pick $\tilde{\Theta}\in \Omega^{n-1}(S^1\times M;V)$ with $-\int_{S^1} \tilde{\Theta}=\Theta$. Then $\int_{S^1}a(\tilde{\Theta})=a(\Theta)$ and by definition (\ref{eqn:mult-odd-1}):
$$ a(\Theta)\times \hat{y}  \overset{(\ref{eqn:mult-odd-1})}{=}  \int_{S^1} a(\tilde{\Theta})\times\hat{y}\overset{(\ref{eqn:partial-a-mult-compatible})}{=}\int_{S^1} a(\tilde{\Theta}\,\overline{\wedge}\, R(\hat{y}))=a\left(-\int_{S^1}\tilde{\Theta}\,\overline{\wedge}\, R(\hat{y}) \right)=a(\Theta\,\overline{\wedge}\, R(\hat{y}))$$
The proof that (\ref{eqn:mult-odd-2}) and (\ref{eqn:mult-odd-3}) are also well-defined and still satisfy (\ref{eqn:aRtimes}) is entirely analogous.
\end{proof}

It is clear that the extended product is bilinear, unital, and graded commutative (cf. (\ref{eqn:mult-odd-3})) in all degrees. For associativity we need a compatibility of multiplication and $\int_{S^1}$:

\begin{proposition}\label{prop:int-mult-comp}
For any $\hat{x}\in \hat{E}^n(S^1\times M)$ and $\hat{y}\in\hat{E}^m(N)$ we have
\begin{equation}\label{eqn:int-mult}
\int_{S^1} (\hat{x}\times \hat{y})=\left(\int_{S^1} \hat{x}\right)\times \hat{y},\hspace{1cm}
(-1)^{|\hat{y}|}\int_{S^1} (\hat{y}\times\hat{x})=\hat{y}\times\left(\int_{S^1} \hat{x}\right)
\end{equation}
\end{proposition}
\begin{proof}
For $n,m$ both even, this is just (\ref{eqn:mult-odd-1}). For $n$ odd and $m$ even we have,
using Lemma \ref{lem:techn2}, which we will prove independently in a moment,
$$ \int_{S^1} (\hat{x}\times \hat{y})\overset{(\ref{eqn:mult-odd-1})}{=}\int_{S^1}\int_{S^1} (\hat{X}\times\hat{y})\overset{(\ref{eqn:double-int-product})}{=}\left(\int_{S^1}\int_{S^1} \hat{X}\right)\times\hat{y} = \left(\int_{S^1} \hat{x}\right)\times \hat{y}$$
For $n$ even and $m$ odd, we have by Proposition \ref{prop:mult-integral}
\begin{align*}
 \int_{S^1}(\hat{x}\times\hat{y})&\overset{(\ref{eqn:mult-odd-2})}{=}\int_{S^1}\int_{S^1}^{'}	 (\hat{x}\times\hat{Y})=-\int_{S^1}\int_{S^1} (\hat{x}\times\hat{Y})
 \overset{(\ref{eqn:mult-odd-1})}{=}-\int_{S^1}\left(\int_{S^1}\hat{x} \right)\times\hat{Y}
 \overset{(\ref{eqn:mult-odd-3})}{=} \left(\int_{S^1} \hat{x}\right)\times \hat{y}
\end{align*}
Finally, for $n,m$ both odd, by applying Proposition \ref{prop:mult-integral} twice,
\begin{align*}
\int_{S^1} (\hat{x}\times\hat{y})&=\int_{S^1}\int_{S^1}\int_{S^1}^{'}(\hat{X}\times\hat{Y})
\overset{(\ref{eqn:mult-integral})}{=}\int_{S^1}\int_{S^1}\int_{S^1}(\hat{X}\times\hat{Y})\overset{(\ref{eqn:double-int-product})}{=}\int_{S^1} \left(\int_{S^1}\int_{S^1} \hat{X}\right)\times \hat{Y}\\
&\overset{(\ref{eqn:mult-odd-2})}{=}\left(\int_{S^1}\int_{S^1} \hat{X}\right)\times \hat{y}=\left(\int_{S^1} \hat{x}\right)\times \hat{y}
\end{align*}
The second asserted equation is then immediate, by graded commutativity.
\end{proof}

It is now easy to verify associativity case by case, using the already proven associativity if all indices are even.

\begin{lemma}\label{lem:techn2}
Suppose $n,m$ are even. For $\hat{z}\in\hat{E}^{n}(S^1\times S^1\times M)$ and $\hat{w}\in \hat{E}^m(N)$ we then have
\begin{equation}
\int_{S^1}\int_{S^1}(\hat{z}\times \hat{w})=\left(\int_{S^1}\int_{S^1} \hat{z}\right)\times\hat{w}\label{eqn:double-int-product}
\end{equation}
\end{lemma}
\begin{proof}
The proof is similar to that of
Proposition \ref{prop:mult-integral} and we will use the notation from there. Represent $\hat{y}$ by a triple $(c_0,\omega_0,h_0)$ as in Lemma \ref{lem:techn-int} that may be integrated directly. Then
for all $\hat{w}=[c_1,\omega_1,h_1]$
\begin{align*}
\int_{S^1}\int_{S^1} (\hat{z}\times \hat{w})=\int_{S^1}\int_{S^1} &(\hat{y}\times\hat{w})
= \Big[ \int_{S^1}\int_{S^1}\mu_{n,m}(c_0\times c_1), \int_{S^1}\int_{S^1}\omega_0\overline{\wedge}\omega_1,\\
&\int_{S^1}\int_{S^1} (B(\omega,\omega_1)+\omega_0\times h_1 + h_0\times \omega_1 - h_0\times \delta h_1 + (c_0\times c_1)^* M_{n,m}  )
   \Big]
\end{align*}
while on the other hand
\begin{align*}
&(  \int_{S^1}\int_{S^1} \hat{z} ) \times \hat{w}=(  \int_{S^1}\int_{S^1} \hat{y}  )\times \hat{w}=\Big[\mu_{n-2,m}\circ ( \int_{S^1}\int_{S^1} c_0\times c_1 ), (\int_{S^1}\int_{S^1} \omega_0)\overline{\wedge} \omega_1,\\
&B( \int_{S^1}\int_{S^1}\omega_0, \omega_1 ) + (\int_{S^1}\int_{S^1}\omega_0)\times h_1 + (\int_{S^1}\int_{S^1} h_0)\times \omega_1 - (\int_{S^1}\int_{S^1} h_0)\times \delta h_1 + (\int_{S^1}\int_{S^1} c_0\times c_1)^*M_{n-2,m} \Big]
\end{align*}
The natural chain homotopies $\int_{S^1}\int_{S^1} B(-,-), B(\int_{S^1}\int_{S^1} -, -)$ are chain homotopic by Acyclic Models. They therefore differ on closed forms only by coboundaries. Select a homotopy $H: I \times \Sigma^2 E_{n-2}\times E_m\rightarrow E_{n+m}$ from
$\mu_{n,m}\circ (\varepsilon_{n-1}\circ \Sigma\varepsilon_{n-2}\times id_m)$ to $\varepsilon_{n+m-1}\circ \Sigma \varepsilon_{n+m-2}\circ \Sigma^2 \mu_{n-2,m}$. It suffices to show that the pullback under $(\int_{S^1}\int_{S^1} c_0) \times c_1$ of
$$ \int_{S^1}\int_{S^1} (\varepsilon_{n-1}\circ\Sigma \varepsilon_{n-2}\times id_m)^* M_{n,m}-\int_I H^*((\Omega \varepsilon_{n+m-1}^{adj})^{-1})^*((\varepsilon_{n+m-2}^{adj})^{-1})^*\iota_{n+m-2}-M_{n-2,m} $$
is a coboundary. As usual, applying $\delta$ yields zero and the result follows. For more details, especially on the use of Acyclic Models here, we refer to \cite{meine_promotion}.
\end{proof}

We close with our main result:

\begin{theorem}\label{thm:main}
Let $E$ be a rationally even, multiplicative cohomology theory. Then there exists a differential extension $\hat{E}^*$ to a multiplicative differential cohomology theory in the sense of \cite{bunke_schick_uniqueness} which is a contravariant functor to the category of unital graded commutative rings. Moreover, the product structure is compatible with integration (Proposition \ref{prop:int-mult-comp}).
\end{theorem}
\begin{proof}
Since pullback $f^*$ along smooth maps $f: M\rightarrow N$ is compatible with integration (Proposition \ref{prop:integration-natural}) it follows from Proposition \ref{prop:pullback-mult} that $f^*$ preserves the product also in odd degrees.
\end{proof}

\newpage

\addcontentsline{toc}{section}{References}
\bibliography{references}
\nocite{*}

\bibliographystyle{amsalpha}

\end{document}